\documentclass{amsart}

\usepackage{amsmath,amsthm,amssymb,amscd,enumerate,mathtools,enumitem}
\usepackage{amsfonts}
\usepackage{rotating}
\usepackage{euscript}
\usepackage{pst-node}
\usepackage{epsfig,verbatim}
\usepackage{stackengine,scalerel}
\stackMath

\def\be{\begin{equation}}
\def\ee{\end{equation}}

\def\C{{\mathbb C}} 
\def\f{\mathcal}
 
\def\P{{\mathbb P}}

\def\ord{{\rm ord\,}}

\def\phi{{\varphi}}
 
\def\tt{\widetilde}
\def\deg{{\rm deg\,}}

\def\GCD{{\rm gcd }}

\def\lcm{{\rm lcm }}

\def\bp{\begin{proposition}}
\def\ep{\end{proposition}}

\def\bt{\begin{theorem}}
\def\et{\end{theorem}}
\def\br{\begin{remark}}
\def\er{\end{remark}}
\def\be{\begin{equation}}
\def\bee{\begin{equation*}}
\def\l{\label}
\def\la{\label}

\def\ee{\end{equation}}
\def\eee{\end{equation*}}
\def\bl{\begin{lemma}}
\def\el{\end{lemma}}
\def\bc{\begin{corollary}}
\def\ec{\end{corollary}}
\def\pr{\noindent{\it Proof. }}

\def\bd{\begin{definition}}
\def\ed{\end{definition}}
\def\t{\widetilde}
\def\tilde{\widetilde}
\def\h{\widehat}

\newtheorem{theorem}{Theorem}[section]
\newtheorem{lemma}[theorem]{Lemma}
\newtheorem{definition}[theorem]{Definition}
\newtheorem{corollary}[theorem]{Corollary}
\newtheorem{proposition}[theorem]{Proposition}
\newtheorem{problem}[theorem]{Problem}

\theoremstyle{definition}

\theoremstyle{definition}
\newtheorem{remark}[theorem]{Remark}
\def\bpr{\begin{problem}}
\def\epr{\end{problem}}

\mathtoolsset{showonlyrefs}

\begin{document}

\title{
Functional equations in branched covering maps
}

\author[F. Pakovich]{Fedor Pakovich}
\thanks{
This research was supported by ISF Grant  No. 1092/22}
\address{Department of Mathematics, Ben-Gurion University of the Negev, Israel}
\email{
pakovich@math.bgu.ac.il}


\begin{abstract}
 We study the functional equation  
$
f \circ p = g \circ q,$   
where \( f \), \( g \), \( p \),  \( q \) are branched covering maps between   closed surfaces. 
In particular, we extend results on semiconjugate holomorphic maps between  compact Riemann surfaces to this broader setting. 
As an application, we present an alternative proof of the fact that every quotient of a torus endomorphism  
has a parabolic orbifold, along with some extensions of this result.
\end{abstract}

\maketitle

\section{Introduction}
The study of the functional equation
\be \l{fu} f \circ p = g \circ q \ee
in rational functions has a long history, originating from the pioneering paper of Ritt (\cite{r1}), who focused on its polynomial solutions. This equation is closely related to the algebro-geometric problem of describing the algebraic curves
\be \l{cu} f(x) - g(y) = 0 \ee
that have a genus-zero factor, which naturally arises in various areas of mathematics, such as number theory, complex analysis, and dynamics (see, e.g., \cite{lo} and the bibliography therein). While a complete description of solutions of \eqref{fu} in rational functions is not yet known, much progress has been made, including reasonable descriptions of commuting rational functions (\cite{ere}, \cite{rev}, \cite{ritt}) and semiconjugate rational functions (\cite{semi}, \cite{dyna}, \cite{rec}, \cite{lattes}, \cite{fin}).

In this paper, we study equation \eqref{fu} under the assumption that the 
maps involved are branched covering maps between closed surfaces, 
where by a surface we always mean an oriented connected two-dimensional topological 
manifold without boundary. In particular, we aim to generalize the results concerning 
semiconjugate holomorphic maps between compact Riemann surfaces from \cite{semi} and \cite{lattes} to this broader 
setting.

The primary motivation for this investigation lies in dynamics. 
Specifically, the Riemann existence theorem implies that if \( f: X \to Y \) is a branched 
covering map between closed surfaces, then for every conformal structure 
on \( Y \), there exists a conformal structure on \( X \) such that 
\( f \) becomes holomorphic. Thus, many properties of an {\it individual}  
branched covering map coincide with those of holomorphic maps. 
However, the {\it dynamics} of branched covering maps exhibit distinctive 
behaviors compared to holomorphic maps, making the study of these 
dynamics a topic of independent interest (see, e.g., the books \cite{bm}, \cite{pil}). 
As one of the most famous results in this context, we mention Thurston's theorem, which characterizes those postcritically finite branched self-covering maps of the sphere that are Thurston equivalent to rational functions (see \cite{dh}, \cite{th}). 

A particularly relevant concept is the semiconjugacy relation  
\be  \l{se}   
f \circ p = p \circ q   
\ee   
for branched covering maps. 
Specifically, in \cite{bm} and \cite{bm2}, the question was raised as to whether 
every branched covering map \( f \) of degree at least two, making the diagram  
\be \l{d} 
\begin{CD}  
T^2 @>q>> T^2 \\  
@VV p V @VV p V \\  
S^2 @>f >> S^2,   
\end{CD}  
\ee   
commute 
for some branched covering maps \( p \) and \( q \), where \( S^2 \) is a sphere 
and \( T^2 \) is a torus, necessarily has a parabolic orbifold (see Section \ref{s32} 
for more details). For holomorphic maps, the 
validity of this statement is well known. However,  this known result does not 
imply its topological analogue,  which reflects the difference between the two settings.

A partial solution to the above question was given in \cite{bm2}, and it was answered in the affirmative in \cite{lx}. In this paper, we refine the results of \cite{bm2} and \cite{lx}, providing a detailed characterization of the solutions to \eqref{se} in branched covering maps, including those that are not of the form \eqref{d}. Furthermore, we offer a unified topological treatment of several results concerning the functional equation \eqref{fu} that are typically established either in the holomorphic setting or in algebro-geometric terms, thereby demonstrating their validity for general branched covering maps.

To formulate our results explicitly, let us introduce some definitions. 
Let \linebreak $p\colon R \to C_1$ and $q\colon R \to C_2$ be branched covering maps between closed surfaces. We say that $p$ and $q$ {\it have no non-trivial common compositional right factor} if the equalities 
\be \l{kaban} 
p = \tt p \circ w, \quad q = \tt q \circ w,
\ee 
where $\tt p\colon \tt R \to C_1$, $\tt q\colon \tt R \to C_2$, and $w\colon R \to \tt R$ are branched covering maps between closed surfaces, imply that $\deg w = 1.$

Let us recall that {\it a surface orbifold} is a pair $\f O=(R,\nu)$ consisting of a surface $R$ and a ramification function $\nu: R \to \mathbb N$, which takes the value $\nu(z)=1$ except at isolated points. For an orbifold $\f O=(R,\nu)$ with closed $R$, the {\it Euler characteristic} of $\f O$ is the number
\be \l{euler} 
\chi(\f O)=\chi(R)+\sum_{z\in R}\left(\frac{1}{\nu(z)}-1\right). 
\ee 
We say that a point $z \in R$ is a ramified point of $\f O$ if $\nu(z) > 1$. The set of ramified points of $\f O$ is denoted by $c(\f O).$ Finally, the signature $\nu(\f O)$ of $\f O$ is defined as the list of all values $\nu(z)$, where $z$ ranges over the ramified points of $\f O$, and each value is included as many times as it occurs among the ramified points.

If $f:\, R_1\rightarrow R_2$ is a branched covering map between surfaces, 
and \linebreak $\f O_1=(R_1,\nu_1)$ and $\f O_2=(R_2,\nu_2)$ are orbifolds,  
then we say that 
$f:\, \f O_1\rightarrow \f O_2$ is a {\it covering map 
between orbifolds} if the equality 
\be \l{us} \nu_{2}(f(z))=\nu_{1}(z)\deg_zf\ee 
holds for all $z\in R_1$, where $\deg_zf$ denotes the local degree of $f$ 
at the point $z$. If, instead of \eqref{us}, the condition 
\be \l{rys} \nu_{2}(f(z))=\nu_ {1}(z)\GCD(\deg_zf, \nu_{2}(f(z)))\ee 
is satisfied for all $z\in R_1$, we say that 
$f:\, \f O_1\rightarrow \f O_2$ is a {\it minimal holomorphic map 
between orbifolds}.  Note that since \eqref{us} implies \eqref{rys}, every 
covering map is also a minimal holomorphic map. 
However, the class of minimal holomorphic maps introduced in \cite{semi} is broader and particularly useful in studying the functional equation \eqref{fu}, as property \eqref{rys} successfully captures the ramification behavior of the maps involved.

 With each branched covering map $f:\, R_1\rightarrow R_2$ between closed     surfaces, 
we associate two orbifolds $\f O_1^f=(R_1,\nu_1^f)$ and 
$\f O_2^f=(R_2,\nu_2^f)$, setting $\nu_2^f(z)$  
equal to the least common multiple of local degrees of $f$ at the points 
of the preimage $f^{-1}\{z\}$, and $$\nu_1^f(z)=\frac{\nu_2^f(f(z))}{\deg_zf}.$$ 
Note that, by construction, $f:\f O_1^f\rightarrow \f O_2^f$ is a covering map between orbifolds. 

For every point \( z \in R_2 \), we also associate with $f:\, R_1\rightarrow R_2$  a partition \( \mu_{f,z} \) of \( \deg f \), defined as the collection of local degrees of \( f \) at the points in \( f^{-1}\{z\} \). 
 Finally, 
for an arbitrary partition  $\mu=(a_1,a_2,\dots, a_k)$ of an integer  $n\geq 1$ and an integer  $d\geq 1$, we define another partition $\mu^d$ of $n$ by replacing each number $a_i,$ $1\leq i \leq k$, with the number ${\rm lcm}(a_i,d)/d$  taken 
$\gcd(a_i,d)$ times.

In this notation, our main result is the following statement.

\bt \l{t1} Let $f,p,g,q$ be branched covering maps between closed surfaces  such  that $\deg f\geq 2,$ 
the diagram 
\be 
\begin{CD}
R @>q>> C_2\\
@VV p V @VV g V\\ 
C_1 @>f >> C\ 
\end{CD}
\ee
commutes,  $p$ and $q$  have no non-trivial common compositional right factor, and $\deg p=\deg g$.  Then for every $z\in C_1$ the equality $$\mu_{p,z}=\mu_{g,f(z)}^{\deg_z f}$$ holds. Moreover, 
$f:\f O_2^p\rightarrow \f O_2^g$ and $q:\f O_1^p\rightarrow \f O_1^g$ are minimal holomorphic maps between orbifolds.  
\et

Theorem \ref{t1} is particularly useful when applied to equation \eqref{se}, where the condition $\deg p = \deg g$ is satisfied automatically. However, Theorem \ref{t1} is applicable only if $p$ and $q$ do not have a non-trivial common compositional right factor, whereas \eqref{se} admits solutions for which this condition fails, as demonstrated by the construction below.

Let $R$ be a closed surface, and $q: R \rightarrow R$ a branched self-covering map of degree at least 
two. Note that by the Riemann-Hurwitz formula, this is possible only if $R$ is a torus or a sphere. For any decomposition $q = v \circ u$, where $u: R \rightarrow R$ 
and $v: R \rightarrow R$ are branched covering maps, we say that the 
branched covering map $q' : R \rightarrow R$, defined by 
$q'  = u \circ v$, is an {\it elementary transformation} of $q$. 
We say that two branched covering maps $q\colon R \to R$ and $f\colon R \to R$ are {\it equivalent} and write $q \sim f$ if there exists a chain of elementary transformations connecting them. 

Note that, for any orientation-preserving homeomorphism $\phi$ of $R$, the equality 
$$q = (q \circ \phi) \circ \phi^{-1}$$ 
implies that
$$q \sim \phi^{-1} \circ q \circ \phi,$$ 
meaning that each equivalence class is a union of conjugacy classes under orientation-preserving  homeomorphisms.

The equivalence relation $\sim$ 
is closely related to the problem of describing  semiconjugate branched covering maps. Specifically, for 
\( q'  \) and $q$ as defined above, we obviously have:
\be \l{one}
q'  \circ u = u \circ q, \quad q \circ v = v \circ q' ,
\ee
which implies by induction that if $q \sim f$, then $q$ is 
semiconjugate to $f$, and $f$ is semiconjugate to $q$. 
Note that in the holomorphic setting, in addition to the problem of describing semiconjugate rational functions, the relation $\sim$ is also naturally related to  other dynamical problems (see \cite{xie, rec, mut, fin, rev}).

It is clear that for $u$ and $v$ of degree greater than one the semiconjugacies defined by the equalities \eqref{one} do not satisfy the assumptions of Theorem \ref{t1}. Nevertheless, Theorem \ref{t1} yields the following general statement, which implies in particular an affirmative answer to the question raised in \cite{bm}, \cite{bm2}. 

For a surface $S$, we denote by $\f S$ the non-ramified orbifold on $S$, that is, the orbifold $(S,\nu)$ with $\nu(z) \equiv 1$.

\bt \l{t2} Let $f,p,q$ be branched covering maps such that $\deg f\geq 2$ and the diagram 
\be 
\begin{CD}
T^2 @>q>> T^2\\
@VV p V @VV p V\\ 
S^2 @>f >> S^2
\end{CD}
\ee
commutes. Then there exists an orbifold of zero Euler characteristic $\f O$ on $S^2$ such that the diagram 
\be 
\begin{CD}
\f T^2 @>q>> \f T^2\\
@VV p V @VV p V\\ 
\f O @>f >> \f O
\end{CD}
\ee
consists of covering maps between orbifolds. 
\et
Notice that since \( \f T^2 \) is unramified, the condition that $p:\f T^2\rightarrow \f O$ is a covering map implies that \( p \) has the same local degree at each point of \( p^{-1}\{z\} \). On the other hand, the condition \( \chi(\f O) = 0 \) is equivalent to \( \f O \) having one of the following signatures: 
\be \l{list} 
(2,2,2,2), \ \ \ (3,3,3), \ \ \ (2,4,4), \ \ \ (2,3,6). 
\ee 
Finally, the condition that \( f : \f O \to \f O \) is a covering map between orbifolds is the  analogue of the corresponding characterization of ordinary Lattès maps.

Let $f:S^2\rightarrow S^2$ be a branched self-covering map. Then the  semiconjugacy relation \eqref{se} does not 
necessarily have the form \eqref{d} described in Theorem \ref{t2}, as \( q \) 
can also be a map between spheres. The following result complements 
Theorem \ref{t2} in this case. 
\bt \label{t3} Let $f,p,q$ be branched covering maps such that $\deg f\geq 2$ and  the diagram 
\be 
\begin{CD}
S^2 @>q>> S^2\\
@VV p V @VV p V\\ 
S^2 @>f >> S^2
\end{CD}
\ee
commutes. Then either $q\sim f$, or there exist orbifolds of non-negative Euler characteristic $\f O_1$ and $\f O_2$ on $S^2$
such that the diagram 
\be 
\begin{CD}
\f O_1 @>q>> \f O_1\\
@VV p V @VV p V\\ 
\f O_2 @>f >> \f O_2
\end{CD}
\ee
consists of minimal holomorphic maps between orbifolds. Furthermore, either \linebreak $\chi(\mathcal{O}_1) = \chi(\mathcal{O}_2) = 0$, or $0 < \chi(\mathcal{O}_2) < \chi(\mathcal{O}_1)$. In the last case, the possible collections of ramification indices of $\mathcal{O}_1$ and $\mathcal{O}_2$ are the following: $(n, n)$ or $(2, 2, n)$ for some $n \ge 2$, or one of the triples $(2, 3, 3)$, $(2, 3, 4)$, $(2, 3, 5)$. In addition, $\mathcal{O}_1$ may be a non-ramified sphere.
\et

Note that in contrast to Theorem \ref{t2}, Theorem \ref{t3} allows for the case $q \sim f$, which cannot occur for diagram \eqref{d}. In addition, the conclusion of Theorem \ref{t2} that $p$, $q$, and $f$ are covering maps between orbifolds is replaced in Theorem \ref{t3} by the weaker condition that they are minimal holomorphic maps. Nevertheless, they can still be covering maps, and indeed they are whenever the Euler characteristic of $\f O_1$ and $\f O_2$ is zero. 


This paper is structured as follows. In Section 2, we review basic 
definitions and results related to branched covering maps, orbifolds, and  maps between orbifolds.  
In Section 3, we present several characterizations of solutions to \eqref{fu} satisfying the condition that $p$ and $q$ have no non-trivial common compositional right factor, which we call {\it reduced} solutions. In particular, we prove a topological version of Abhyankar's lemma, which is usually formulated 
in terms of field extensions.

It is evident that the study of solutions to \eqref{fu} reduces to the analysis of reduced solutions. However, these solutions form a very broad class, and the number of general results concerning them is limited. Much more explicit results can be obtained by imposing the additional condition $\deg p = \deg g$. We call such solutions to \eqref{fu} \textit{good} and examine them in Section 4. In particular, we prove Theorem \ref{t1}.

In Section 5, we study solutions to the semiconjugacy equation \eqref{se}. First, we show how an arbitrary solution can be reduced to a solution where $p$ and $q$ do not have a non-trivial common compositional right factor. We call such solutions to \eqref{se} {\it primitive}. Then, using Theorem \ref{t1} and properties of minimal holomorphic maps between orbifolds, we prove Theorems \ref{t2} and \ref{t3}. We also explore the connection between Theorem \ref{t2} and the problem  studied in \cite{bm}, \cite{bm2}, and \cite{lx}.

In Section 6, we prove several results about 
arbitrary solutions to \eqref{fu}. 
We first describe solutions to 
\eqref{fu} in terms of fiber products and the diagonal monodromy action. We then discuss branched covering maps that are Galois coverings, and the process of normalization, which corresponds to the passage from a function field to its Galois closure in the algebro-geometric setting.

 Finally, we prove  in the setting of 
branched covering maps an analogue of Fried's result concerning the 
reducibility of algebraic curves of the form \eqref{cu}, where $f$ and $g$ are rational functions.

\section{\la{prlm} Preliminaries}
 
\subsection{ Branched covering maps and orbifolds} 
We begin by recalling the definition of a branched covering map (for more  details, see \cite{bm}). 
A continuous  map $f:R\to C$ between surfaces is called a \emph{branched
covering map} if every point $y\in C$ has an open
neighborhood $V$ such that
\[
f^{-1}(V)=\bigsqcup_{i\in I}U_i,
\]
where $I$ is nonempty and the sets $U_i$ are pairwise
disjoint open subsets of $R$, satisfying the following
conditions. Each $U_i$ contains exactly one point
$x_i\in f^{-1}\{y\}$, and there exist orientation-preserving
homeomorphisms
\[
\varphi_i:U_i\to\mathbb D,\qquad
\psi_i:V\to\mathbb D,
\]
where $\mathbb D$ is the open unit disk, such that
\[
\varphi_i(x_i)=\psi_i(y)=0,\qquad
(\psi_i\circ f\circ\varphi_i^{-1})(z)=z^{d_i},
\quad z\in\mathbb D,
\]
for some integer $d_i\geq1$. 

The integer $d_i$ is independent of the chosen neighborhoods
and coordinates. It is called the \emph{local degree} of
$f$ at $x_i$ and is denoted by $\deg_{x_i}f$.
A point $x\in R$ is called a \emph{critical point} of $f$
if $\deg_x f\geq2$, and its image $f(x)$ is called
a \emph{critical value}.
Every branched covering map $f:R\to C$ is surjective, open, and discrete. If \( f: R \to C \) is a branched covering map between surfaces, then for any conformal structure on \( C \) there exists a unique conformal structure on \( R \) such that \( f \) becomes holomorphic. This fact allows us to derive many results for branched covering maps from their holomorphic counterparts.

A branched covering map \( f: R \to C \) between {\it closed} surfaces is finite-to-one, and the set of its critical values is finite. We denote this set by \( S_f \). If \( S \) is a finite set containing \( S_f \), then the restriction
$
f: R \setminus f^{-1}(S) \longrightarrow C \setminus S
$
is an ordinary finite covering map. Conversely, if $C$ is a closed surface, $S \subset C$ is finite, and $f: R' \to C \setminus S$ is a finite covering map, where $R' = R \setminus P$ for some closed surface $R$ and finite set $P \subset R$, then $f$ extends uniquely to a branched covering map $f: R \to C$ with critical values contained in $S$ (see,  e.g., \cite{kho}). In view of this correspondence, we will use the same symbol $f$ for both the branched covering map and its restriction, which is a covering map, specifying the domain and codomain when necessary.

Recall that {\it a universal covering} of an orbifold $\mathcal{O}=(R,\nu)$ is a covering map between orbifolds $\theta_{\mathcal{O}} : \tilde{\mathcal{O}} \to \mathcal{O}$, where $\tilde{\mathcal{O}}=(\t R,\t \nu)$ is an orbifold such that $\tilde{R}$ is simply connected and $\t \nu(z)\equiv 1$.  A universal covering exists and is unique up to an orientation-preserving homeomorphism of $\tilde{R}$, unless $R = S^2$ with a single ramified point, or $R = S^2$ with two ramified points $z_1, z_2$ such that $\nu(z_1) \neq \nu(z_2)$. We say that an orbifold $\mathcal{O}$ is good if it has a universal covering, and bad otherwise. 

We assume that all orbifolds considered in this paper are good.
In particular, the orbifolds $\f O_1^f$ and $\f O_2^f$ defined in
the introduction are good: the proof of \cite[Lemma~4.2]{semi}
carries over to the topological setting.

Now let us make some comments on the definition of a minimal holomorphic map from the introduction.
Assume for a moment that \( R_1 \) and \( R_2 \) are Riemann surfaces 
and let  \( f: R_1 \to R_2 \) be a holomorphic branched covering map. 
Furthermore, let  \( \f O_1 = (R_1, \nu_1) \) and \( \f O_2 = (R_2, \nu_2) \) be orbifolds, and \(\theta_1: \t{ R}_1 \rightarrow R_1 \) and \linebreak \(\theta_2: \t{ R}_2 \rightarrow R_2 \) their universal coverings. 
It is easy to see that if $f:\f O_1\rightarrow \f O_2$ is a covering map, as defined by equality \eqref{us}, then $f$ can be lifted to a conformal isomorphism $\t f$ which makes the diagram 
\be  \l{dsa} 
\begin{CD}
\t R_1 @>\t f>> \t R_2\\
@VV  \theta_1 V @VV  \theta_2 V\\ 
R_1 @>f >> R_2 
\end{CD}
\ee
commutative. 
Along with condition \eqref{us}, 
it is useful to consider the weaker condition that 
\be \l{uuss} 
\nu_2(f(z)) \mid \nu_1(z) \deg_z f 
\ee for any \( z \in R_1 \). In this case, we say that \( f: \f O_1 \to \f O_2 \) is a {\it holomorphic map} between orbifolds. For such a map, it is still possible to lift $f$ to some map $\t f$ making the diagram \eqref{dsa} commutative, which is, however, merely holomorphic and not necessarily an isomorphism. This explains the term ``holomorphic map'' between orbifolds, and we adopt this definition, even though our setting is purely topological.

Note that in the construction of an orbifold associated with a postcritically finite rational map due to Thurston, the condition 
\be 
 \nu_1(z) \deg_z f \mid \nu_2(f(z))
\ee
opposite to \eqref{uuss} appears. 
Correspondingly, this condition ensures that one can lift instead of $f$ a branch of $f^{-1}.$

Returning to the topological setting, assume that \( f: R_1 \to R_2 \) is a branched covering map between closed surfaces where \( R_2 \) is equipped with a ramification function \( \nu_2 \). Then defining a ramification function \( \nu_1 \) on \( R_1 \) such that \( f \) becomes a holomorphic map between the orbifolds \( \f O_1 = (R_1, \nu_1) \) and \( \f O_2 = (R_2, \nu_2) \) requires satisfying condition \eqref{uuss}, and it is evident that for any \( z \in R_1 \), the minimal possible value of \( \nu_1(z) \) is given by equality \eqref{rys}, which justifies the term ``minimal holomorphic map.''

It follows from the definition of a minimal holomorphic map that for any branched covering map $f\colon R' \to R$ between closed surfaces and for any orbifold $\f O = (R, \nu)$, there exists a unique orbifold structure $\nu'$ on $R'$ such that $f$ becomes a minimal holomorphic map between the orbifolds $\f O' = (R', \nu')$ and $\f O = (R, \nu)$. We denote the corresponding orbifold $\f O'$ by $f^*\f O$. 
Note that for any minimal holomorphic map $f\colon \f O_1 \to \f O_2$, the equality \be \l{zzxx} \f O_1 = f^*(\f O_2) \ee holds simply by definition. In particular, equality \eqref{zzxx} holds for any covering map between orbifolds $f\colon \f O_1 \to \f O_2$.

Minimal holomorphic maps between orbifolds possess the following fundamental property (see \cite{semi}, Theorem 4.1). Although it was originally proved in the holomorphic setting, the proof is based on local calculations and carries over verbatim to the topological setting. 

\bt \l{serrr} Let $f:\, R_1 \rightarrow R$ and $g:\, R \rightarrow R_2$ be  branched covering maps  between closed surfaces, and  $\f O=(R_2,\nu)$ an orbifold. 
Then 
$$(g\circ f)^*(\f O)= f^*(g^*(\f O)).\eqno{\Box}$$
\et

Theorem \ref{serrr} implies in particular the following corollaries (see   \cite{semi}, Corol- \linebreak lary 4.1 and Corollary 4.2).

\bc \l{serka0}   Let $f:\, R_1 \rightarrow R$ and $g:\, R \rightarrow R_2$ be  branched covering maps between closed surfaces, and  $\f O_1=(R_1,\nu_1)$, $\f O_2=(R_2,\nu_2)$, and $\f O=(R,\nu)$ 
orbifolds. Assume that  $f:\, \f O_1\rightarrow \f O$ and $g:\, \f O\rightarrow \f O_2$ are minimal holomorphic maps (resp. covering maps) between orbifolds.
Then  $g\circ f:\, \f O_1\rightarrow \f O_2$ is  a minimal holomorphic map (resp. covering map)  between orbifolds. \qed
\ec

\bc \l{indu2}  Let $f:\, R_1 \rightarrow R$ and $g:\, R \rightarrow R_2$ be  branched covering maps  between closed surfaces, and  $\f O_1=(R_1,\nu_1)$ and  $\f O_2=(R_2,\nu_2)$
orbifolds. Assume that  $g\circ f:\, \f O_1\rightarrow \f O_2$ is  a minimal holomorphic map between orbifolds (resp. a co\-vering map). Then  $g:\, g^*(\f O_2)\rightarrow \f O_2$  and  $f:\, \f O_1\rightarrow g^*(\f O_2) $ are minimal holomorphic maps (resp. covering maps)  between orbifolds. \qed
\ec

If \( f: R_1 \to R_2 \) is a branched covering map between closed surfaces and \linebreak \( \f O_1 = (R_1, \nu_1) \) and \( \f O_2 = (R_2, \nu_2) \) are orbifolds
such that \( f: \f O_1 \to \f O_2 \) is a covering map between orbifolds, 
then the Riemann-Hurwitz formula implies that
\be \l{ner1} 
\chi(\f O_1) = d \chi(\f O_2),
\ee
where \( d = \deg f \). More generally, if \( f: \f O_1 \to \f O_2 \) is a holomorphic map  between orbifolds, 
then
\be \l{ner2}
\chi(\f O_1) \leq \chi(\f O_2) \deg f,
\ee
and equality is attained if and only if \( f: \f O_1 \to \f O_2 \) is a covering map between orbifolds.  The last  statement was proved in \cite{semi}, Proposition 3.2, for holomorphic branched covering maps, but its proof is purely topological and therefore remains valid for arbitrary branched covering maps.

Note that if $\f O=(R,\nu)$ is an orbifold and $f:R\rightarrow R$ is a branched self-covering map of degree at least two such that \( f: \f O \to \f O \) is a covering map between orbifolds, then \eqref{ner1} implies that $\chi(\f O)=0$, whereas \eqref{ner2} implies that if \( f: \f O \to \f O \) is a holomorphic map between orbifolds, then $\chi(\f O)\geq 0$.

\subsection{Branched covering maps and fundamental groups} 
From now on, all surfaces are assumed to be closed, and
instead of ``a branched covering map \linebreak $h:R\to C$ between
closed surfaces'', we will usually simply write
``a branched covering map $h:R\to C$''.

Let $C$ be a surface, $S \subset C$ a finite set, and $z_0$ a point in $C \setminus S$. Recall that for any branched covering map $h: R \to C$  that is unramified outside $S$, and a point $e \in h^{-1}\{z_0\}$, the homomorphism of the fundamental groups 
$$h_{\star}: \pi_1(R \setminus h^{-1}(S), e) \rightarrow \pi_1(C \setminus S, z_0)$$ 
is a monomorphism whose image $\Gamma_{h,e}$ is a subgroup of finite index in $\pi_1(C \setminus S, z_0)$. Conversely, if $\Gamma$ is a subgroup of finite index in $\pi_1(C \setminus S, z_0)$, then there exist a surface $R$, a branched covering map $h: R \rightarrow C$, and a point $e \in h^{-1}\{z_0\}$ such that 
$$h_{\star}(\pi_1(R \setminus h^{-1}(S), e)) = \Gamma.$$  
Furthermore, this correspondence descends to a one-to-one correspondence between conjugacy classes of subgroups of index $d$ in $\pi_1(C \setminus S, z_0)$ and equivalence classes of branched covering maps of degree $d$ unramified outside $S$, where two maps \linebreak $h: R \rightarrow C$ and $\tilde{h}: \tilde{R} \rightarrow C$ are considered equivalent if there exists a homeomorphism $w: R \rightarrow \tilde{R}$ such that $h = \tilde{h} \circ w$. To distinguish this equivalence relation from the one defined in the introduction, we will always use the symbol $\sim$ when referring to the latter.  

If $h:R\to C$ and $f:C_1\to C$ are branched covering
maps and $e\in h^{-1}\{z_0\}$, $c\in f^{-1}\{z_0\}$
are points, then the inclusion
$\Gamma_{h,e}\subseteq\Gamma_{f,c}$ holds if and only if
there exists a branched covering map $p:R\to C_1$
such that $h=f\circ p$ with $p(e)=c$, and if such
a map exists, it is unique. All the above facts follow from the Galois correspondence between coverings of $C\setminus S$ with a marked point and subgroups of the group $\pi_1(C\setminus S, z_0)$, together with the fact that any finite covering of $C\setminus S$ extends uniquely to a branched covering of $C$ (see, e.g., \cite{kho}). 

We now translate the above description into the language of permutation groups. Let $G$ be  
a group acting transitively on a finite set $X$. We recall that a nonempty subset $B \subset X$ is called a {\it block} if for every $g \in G$, either $g(B) = B$ or $g(B) \cap B = \emptyset$.  The {\it trivial} blocks are the singletons and the whole set $X$. Given a block $B$, the set of all distinct translates
$
\mathcal{B} = \{ g(B) \mid g \in G \}
$
forms a partition of $X$ called a \emph{system of imprimitivity}. It follows from the definition that for a block $B$, the group  
\[
G_B = \{ g \in G \mid g(B) = B \}
\]
contains the stabilizer of any point $x \in B$. Furthermore, there is a bijection between blocks containing a point $x$ and subgroups containing $G_x$. Namely, for any subgroup $H$ with $G_x \le H \le G$, the orbit $Hx$ is a block containing $x$ (see, e.g., \cite{mor}). 

Under the above notation, the group \( \pi_1(C \setminus S, z_0) \) naturally acts on the set \( h^{-1}\{z_0\}, \) and the corresponding permutation group is called the monodromy group of $h$. We denote this group by \( G_h \). For \( \gamma \in \pi_1(C \setminus S, z_0) \), we denote by \( \gamma_h \) the permutation corresponding to \( \gamma \) in the group \( G_h \). Since the group \( \Gamma_{h, e} \) is the stabilizer of \( e \) under the action of \( \pi_1(C \setminus S, z_0) \) on \( h^{-1}\{z_0\} \), the action of \( G_h \) on \( h^{-1}\{z_0\} \)  can be identified with  the action of \( \pi_1(C \setminus S, z_0) \) on the left cosets of \( \Gamma_{h, e} \) by left multiplication.

If a branched covering map $h: R \rightarrow C$ of degree $d$ can be decomposed into a composition $h = f \circ p$ of branched covering maps $p: R \rightarrow C_1$ and $f: C_1 \rightarrow C$ of degrees at least $2$, then the group $G_{h}$ has a  system of imprimitivity $\Omega_f$ consisting of $d_1 = \deg f$ blocks such that the action of $G_h$ on the blocks of $\Omega_f$ is isomorphic to $G_f$. Specifically, the blocks of $\Omega_f$ have the form $$\mathcal{B}_i = p^{-1}\{t_i\},\qquad 1 \leq i \leq d_1,$$ where 
$$\{t_1,t_2,\dots, t_{d_1}\} = f^{-1}\{z_0\}.$$ 
Moreover, each system of imprimitivity  $\Omega$ of $G_h$ consisting of non-trivial blocks has such a form. Indeed, the stabilizer of a block containing $e \in h^{-1}\{z_0\}$ gives a subgroup of $\pi_1(C \setminus S, z_0)$ of the form $\Gamma_{f,c}$ containing $\Gamma_{h,e}$. Thus, $h = f \circ p$ for some branched covering map $p$, and one can easily see that the blocks of $\Omega$ are precisely the fibers $p^{-1}\{t_i\}$ over the points $t_i \in f^{-1}\{z_0\}$.

\section{Reduced solutions of \( f \circ p = g \circ q \)}
\subsection{A characterization of reduced solutions}
In this section, we prove several results about branched covering maps  that make the diagram 
\be \l{dia} 
\begin{CD}
R @>q>> C_2\\
@VV p V @VV g V\\ 
C_1 @>f >> C 
\end{CD}
\ee
commutative, under the condition that $p$ and $q$ have no non-trivial common compositional right factor, as defined in the introduction.   
We will refer to such quadruples $f,g,p,q$ as {\it reduced} solutions of \eqref{dia}, using the terminology of \cite{lo} originally introduced for holomorphic maps between compact Riemann surfaces.   
Notice that the commutativity of diagram \eqref{dia}   
implies that for every $t_1 \in C_1$, the map $q$ sends the fiber $p^{-1}\{t_1\}$ to the fiber $g^{-1}\{f(t_1)\}$, regardless of the nature of the maps $f, g, p, q$.

In the case where $p: R \to C_1$ and $q: R \to C_2$ are holomorphic maps between compact Riemann surfaces, the condition that $p$ and $q$ have no non-trivial common compositional right factor is equivalent to the requirement that the compositum of the subfields $p^*\f M(C_1)$ and $q^*\f M(C_2)$ of $\f M(R)$ is the whole field $\f M(R)$, where $\f M(T)$ denotes the field of meromorphic functions on a compact Riemann surface $T$. This equivalence has no analogue for arbitrary branched covering maps $p$ and $q$. However, if there exist branched covering maps $f$ and $g$ that make diagram \eqref{dia} commutative, then the property that $p$ and $q$ have no non-trivial common compositional right factor can be described in topological terms as follows. 

Set  
\be \l{ic} h = f \circ p = g \circ q.\ee 
For a point $t_0 \in R \setminus h^{-1}(S_h)$,   consider the groups $\Gamma_{f,p(t_0)}$, $\Gamma_{g,q(t_0)}$, and $\Gamma_{h,t_0}$. Clearly, we have 
\be \l{aqe} \Gamma_{h,t_0}\subseteq \Gamma_{f,p(t_0)}\cap \Gamma_{g,q(t_0)}.\ee 

\bp \l{l1} Let $f,p,g,q$ be branched covering maps between closed surfaces such that diagram \eqref{dia} commutes, and let $h=f\circ p=g\circ q$. Then the following conditions are equivalent. 

\begin{enumerate}[label=\arabic*\textup{)}]
\item The maps $p$ and $q$ have no non-trivial common compositional right factor. 
\item For all $t_0\in R\setminus h^{-1}(S_h)$, the equality \be \l{eqa} \Gamma_{f,p(t_0)}\cap \Gamma_{g,q(t_0)}= \Gamma_{h,t_0}\ee holds. 
\item For all $t_1\in C_1\setminus f^{-1}(S_h)$, the map $q$ sends the fiber $p^{-1}\{t_1\}$ to the fiber $g^{-1}\{f(t_1)\}$ injectively. 
\item For all $t_0\in R\setminus h^{-1}(S_h)$, the equality $$p^{-1}\{p(t_0)\}\cap q^{-1}\{q(t_0)\}=\{t_0\}$$ holds. 
\end{enumerate}  
\ep 

\pr 
The equivalence $1) \Leftrightarrow 2)$ follows from the Galois correspondence between coverings of $C\setminus S_h$ with a marked point  $z_0 = h(t_0)$ and subgroups of the group $\pi_1(C\setminus S_h, z_0)$.  

The equivalence $3) \Leftrightarrow 4)$ is clear, since for  $t$ satisfying $p(t) = p(t_0)$ the condition 
$q(t) = q(t_0)
$ 
is equivalent to the condition that $t$ belongs to the set 
\[
T = \mathcal{P}_{t_0}\cap \mathcal{Q}_{t_0},
\] 
where 
\[
\f P_{t_0} = p^{-1}\{p(t_0)\} \quad \text{and} \quad \f Q_{t_0} = q^{-1}\{q(t_0)\}.
\]  

Let us prove the equivalence $2) \Leftrightarrow 4)$.   
Let $\Gamma$ be the subgroup of $\pi_1(C\setminus S_h,z_0)$ which stabilizes $T$.   
Since  $\f P_{t_0}$ and $\f Q_{t_0}$
are blocks with respect to the action of $\pi_1(C\setminus S_h,z_0)$, and 
$\Gamma_{f, p(t_0)}$ and $\Gamma_{g, q(t_0)}$ are subgroups of $\pi_1(C\setminus S_h,z_0)$ which stabilize $\f P_{t_0}$ and $\f Q_{t_0}$, we have  
$$\Gamma= \Gamma_{f,p(t_0)}\cap \Gamma_{g,q(t_0)}.$$ 
Thus, to complete the proof, it is sufficient to show that \( \Gamma_{h,t_0} = \Gamma \) if and only if \( T = \{ t_0 \} \). 

Assume that \( T \neq \{ t_0 \} \).   
Since \( \pi_1(C \setminus S_h, z_0) \) acts transitively on \( h^{-1}\{z_0\} \), for any \( t_1 \in T \) distinct from \( t_0 \), there exists \( \gamma \in \pi_1(C \setminus S_h, z_0) \) such that \( \gamma_h(t_0) = t_1 \). Such a \( \gamma \) does not belong to \( \Gamma_{h,t_0} \), but stabilizes both \( \mathcal{P}_{t_0} \) and \( \mathcal{Q}_{t_0} \) and hence belongs to \( \Gamma \). Thus, \( \Gamma_{h,t_0} \neq \Gamma \).

On the other hand, if \( \Gamma_{h,t_0}\neq \Gamma \), then any element of \( \Gamma \setminus \Gamma_{h,t_0} \) maps \( t_0 \) to some \( t \in T \) distinct from \( t_0 \). Therefore, in this case \( T \) cannot consist of \( t_0 \) alone. \qed 

\subsection{Generalized Abhyankar's lemma} 
For a point \( t_0 \in R \setminus h^{-1}(S_h) \) and \linebreak \( \gamma \in \pi_1(C \setminus S_h, z_0) \), where \( z_0 = h(t_0) \), define \( \ord_{t_0}\gamma_h \) as the smallest integer \( k \geq 1 \) such that \( (\gamma_h)^k(t_0) = t_0 \). 
Similarly, for \( t_1 \in f^{-1}(z_0) \) and   \( t_2 \in  g^{-1}(z_0) \), define \( \ord_{t_1}\gamma_f \) and \( \ord_{t_2}\gamma_g \) as the smallest integers \( k_1 \geq 1 \) and 
\( k_2 \geq 1 \)
such that \( (\gamma_f)^{k_1}(t_1) = t_1 \) and \( (\gamma_g)^{k_2}(t_2) = t_2 \), respectively. 
In this notation, the following statement holds.

\bt \l{t4} Let $f,p,g,q$ be branched covering maps between closed surfaces such that the diagram 
\be \l{buri} 
\begin{CD}
R @>q>> C_2\\
@VV p V @VV g V\\ 
C_1 @>f >> C 
\end{CD}
\ee
commutes, and $p$ and $q$ have no non-trivial common compositional right factor. Then for the branched covering map $h=f\circ p=g\circ q$, the equality  
\be \l{abj0} \ord_{t_0}\gamma_h={\rm lcm} \big(\ord_{p(t_0)}\gamma_f, \ord_{q(t_0)}\gamma_g\big) \ee 
holds for all $t_0\in R \setminus h^{-1}(S_h)$ and $\gamma\in \pi_1(C \setminus S_h, z_0)$, where $z_0=h(t_0)$.
\et

\pr  
Since the actions of the groups \( G_f \), \( G_g \), and \( G_h \) on \( f^{-1}\{z_0\} \), \( g^{-1}\{z_0\} \), and \( h^{-1}\{z_0\} \) can be identified with the action of \( \pi_1(C \setminus S_h, z_0) \) on the left cosets of \( \pi_1(C \setminus S_h, z_0) \) by   \( \Gamma_{f, p(t_0)} \), \( \Gamma_{g, q(t_0)} \), and \( \Gamma_{h, t_0} \) via left multiplication, the statement of the theorem is equivalent to the following.   
For $\gamma\in\pi_1(C\setminus S_h,z_0)$, let $k,k_1,k_2$
be the least positive integers such that
$\gamma^k\in\Gamma_{h,t_0}$,
$\gamma^{k_1}\in\Gamma_{f,p(t_0)}$, and
$\gamma^{k_2}\in\Gamma_{g,q(t_0)}$, respectively.
Then $k=\lcm(k_1,k_2)$.  
This last statement follows from the equivalence $1) \Leftrightarrow 2)$ in Proposition \ref{l1}. 
\qed 

In the case where the  maps involved are holomorphic maps between compact Riemann surfaces, the corollary of Theorem \ref{t4} presented below is known as Abhyankar's lemma. It can be expressed in terms of algebraic curves, or equivalently, in terms of field extensions (see \cite{sti}, Theorem 3.9.1). 
Since any non-constant holomorphic map between compact
Riemann surfaces is a branched covering map, the topological proof of Theorem \ref{t4} given above provides a proof of Abhyankar's lemma over \( \mathbb{C} \).

\bc \l{co1} Let $f,p,g,q$ be branched covering maps between closed surfaces such that the diagram \eqref{buri} commutes, and $p$ and $q$ have no non-trivial common compositional right factor. Then for the branched covering map $h=f\circ p=g\circ q$, the equality  
\be \l{abj} \deg_{t}h={\rm lcm} \big(\deg_{p(t)}f, \deg_{q(t)}g\big) \ee 
holds for all $t\in R$. Equivalently, 
\be \l{abj+} \gcd(\deg_{t}p,\deg_{t}q)=1\ee  for all $t\in R$.   In particular, $S_h=S_f\cup S_g$.
\ec

\pr  To prove \eqref{abj} it suffices to observe that \( \deg_{p(t)}f \), \( \deg_{q(t)}g \), and \( \deg_{t}h \) equal, respectively, \( \ord_{p(t_0)}\gamma_f \), \( \ord_{q(t_0)}\gamma_g \), and \( \ord_{t_0}\gamma_h \), where $t_0$ is a point in  $R\setminus h^{-1}(S_h)$ close enough to $t$ and \( \gamma \in \pi_1(C \setminus S_h, h(t_0)) \) is a small loop around \( z = h(t) \).

The equivalence of \eqref{abj} and \eqref{abj+} follows from the equality 
$$
\deg_{t} h=\deg_{t} q \cdot \deg_{q(t)} g = \deg_{t} p \cdot \deg_{p(t)} f.
$$

Finally, by the chain rule, \[S_f \cup S_g \subseteq S_h, \] and \eqref{abj} implies that \( \deg_th=1 \) unless \( h(t) \) is a critical value of \( f \) or \( g \). \qed 

\section{Good solutions of \( f \circ p = g \circ q \)}
\subsection{A characterization of good solutions}
We start by characterizing branched covering maps  $f,p,g,q$ that make diagram 
\eqref{dia} commutative and, in addition to the property that 
$p$ and $q$ have no non-trivial common compositional right factor, 
satisfy the condition $\deg p = \deg g$ or the equivalent condition  $\deg f = \deg q$.  We will refer to such 
quadruples  $f,p,g,q$ as {\it good} solutions of \eqref{fu}. This definition is consistent with the definition of good solutions of \eqref{fu} in the holomorphic setting  (see \cite{semi} and Corollary \ref{21} below), and we keep this terminology.

\bp \l{l2}  Let $f,p,g,q$ be branched covering maps between closed surfaces such that diagram \eqref{dia} commutes, and let $h=f\circ p=g\circ q$.  Then the following conditions are equivalent. 

\begin{enumerate}[label=\arabic*\textup{)}]
\item The maps $p$ and $q$ have no non-trivial common compositional right factor and $\deg g = \deg p$.   
\item For all $t_0\in R\setminus h^{-1}(S_h)$ the equalities  
\[
\Gamma_{f,p(t_0)}\cap \Gamma_{g,q(t_0)} = \Gamma_{h,t_0}
\] 
and 
\[
\Gamma_{f,p(t_0)}\Gamma_{g,q(t_0)} = \Gamma_{g,q(t_0)}\Gamma_{f,p(t_0)} = \pi_1(C\setminus S_h, h(t_0))
\] 
hold. 
\item For all $t_1\in C_1\setminus f^{-1}(S_h)$ the map $q$ sends the fiber $p^{-1}\{t_1\}$ bijectively onto the fiber $g^{-1}\{f(t_1)\}$. 
\item For all $t_1\in C_1\setminus f^{-1}(S_h)$ and $t_2\in C_2\setminus g^{-1}(S_h)$ with $f(t_1)=g(t_2)$, the intersection \be \l{inte} p^{-1}\{t_1\}\cap q^{-1}\{t_2\} \ee consists of exactly one element.
\end{enumerate}    
\ep 

\pr 
Set  
$
n = \deg f,$ $ m = \deg g,  
$  
and $z_0 = h(t_0)$ for $t_0\in R\setminus h^{-1}(S_h)$. 
Observe that $\deg g = \deg p$ is equivalent to $\deg h = mn$,  
which in turn is equivalent to  
\be \label{z0}  
[\pi_1(C \setminus S_h, z_0) : \Gamma_{h, t_0}] = nm.  
\ee  
Since \eqref{aqe} implies  
\[
[\pi_1(C \setminus S_h, z_0) : \Gamma_{f, p(t_0)} \cap \Gamma_{g, q(t_0)}]  
\leq [\pi_1(C \setminus S_h, z_0) : \Gamma_{h, t_0}],
\]  
and equality is attained if and only if \eqref{eqa} holds, it follows from Proposition \ref{l1}  
that the first condition of the proposition  is equivalent to equality \eqref{eqa} supplemented by the equality 
\be \label{z}  
[\pi_1(C \setminus S_h, z_0) : \Gamma_{f, p(t_0)} \cap \Gamma_{g, q(t_0)}] = nm.  
\ee  
Furthermore, since  
\begin{multline}  
\left[\pi_1(C \setminus S_h, z_0) : \Gamma_{f, p(t_0)} \cap \Gamma_{g, q(t_0)} \right] = \\  
\left[\pi_1(C \setminus S_h, z_0) : \Gamma_{g, q(t_0)} \right]  
\left[\Gamma_{g, q(t_0)} : \Gamma_{f, p(t_0)} \cap \Gamma_{g, q(t_0)} \right],  
\end{multline}  
equality \eqref{z} is equivalent to the equality 
\be \label{z3}  
\left[\Gamma_{g, q(t_0)} : \Gamma_{f, p(t_0)} \cap \Gamma_{g, q(t_0)} \right] = n.  
\ee  

Recall that for any subgroups \( A \) and \( B \) of finite index in a group \( G \), the inequality  
\be \label{ggh}  
\left[\langle A, B \rangle : A \right] \geq \left[ B : A \cap B \right]  
\ee  
holds, and equality is attained if and only if \( A \) and \( B \) are permutable  
(see, e.g., \cite{kur}, p. 79). Therefore,  
\begin{multline}  
n = \left[\pi_1(C \setminus S_h, z_0) : \Gamma_{f, p(t_0)} \right] \geq \\  
\left[\langle \Gamma_{f, p(t_0)}, \Gamma_{g, q(t_0)} \rangle : \Gamma_{f, p(t_0)} \right] \geq  
\left[\Gamma_{g, q(t_0)} : \Gamma_{f, p(t_0)} \cap \Gamma_{g, q(t_0)} \right],   
\end{multline}  
and hence equality \eqref{z3} holds if and only if  
\( \Gamma_{f, p(t_0)} \) and \( \Gamma_{g, q(t_0)} \) are permutable and generate  
\( \pi_1(C \setminus S_h, z_0) \). This proves \( 1) \Leftrightarrow 2) \).  

The equivalence \( 1) \Leftrightarrow 3) \) follows directly from the corresponding equivalence in Proposition \ref{l1} and $\deg g = \deg p$.

Finally, we prove \( 1) \Leftrightarrow 4) \). If \( 1) \) holds, then by Proposition \ref{l1}  the intersection \eqref{inte} contains at most one element. Therefore, the block $p^{-1}\{t_1\}$ intersects exactly $|p^{-1}\{t_1\}|$ blocks $q^{-1}\{t\}$ with $g(t)=f(t_1)$. Since there are $\deg g$ such blocks, and \[ \deg g = \deg p = |p^{-1}\{t_1\}|, \] it follows that the intersection \eqref{inte} is non-empty. Thus, \( 1) \Rightarrow 4) \).

Conversely, \( 4) \) implies by Proposition \ref{l1}  that \( p \) and \( q \) have no non-trivial common compositional right factor. Moreover, \( 4) \) implies that the total number of elements in $h^{-1}\{z_0\}$ is $nm$, so $\deg g = \deg p$. Hence, \( 4) \Rightarrow 1) \). \qed

\subsection{Proof of Theorem \ref{t1}} 
Let $f,p,g,q$ be branched covering maps  that  satisfy the conditions of Proposition \ref{l2}, and let $z_0\in C\setminus S_h$. Then any labelings of the points in the fibers $f^{-1}\{z_0\}$ and $g^{-1}\{z_0\}$ naturally induce a labeling of the points in the fiber $h^{-1}\{z_0\}$, and the permutation $\gamma_h$ can be reconstructed from the permutations $\gamma_f$ and $\gamma_g$ as follows.

For \( z_0 \in C \setminus S_h \), fix a labeling \( \{a_1, a_2, \dots, a_n\} \) of the elements of \( f^{-1}\{z_0\} \) and a labeling \( \{b_1, b_2, \dots, b_m\} \) of the elements of \( g^{-1}\{z_0\} \). Now introduce a labeling \( \{c_{i,j}\} \), \( 1 \leq i \leq n \), \( 1 \leq j \leq m \), of the elements of $h^{-1}\{z_0\}$ by setting
\be \l{per} \{c_{i,j}\}= p^{-1}\{a_i\}\cap q^{-1}\{b_j\}, \quad 1 \leq i \leq n, \quad 1 \leq j \leq m. \ee
By the fourth condition of Proposition \ref{l2}, such a labeling is well-defined. 
In this way, \( \gamma_h \) is identified with a permutation of  
\( c_{i,j} \), \( 1 \leq i \leq n \), \( 1 \leq j \leq m \).

Now for \( \gamma \in \pi_1(C\setminus S_h,z_0) \), define a permutation \( \gamma_{f,g} \) acting on the set \( c_{i,j} \), \( 1 \leq i \leq n \), \( 1 \leq j \leq m \), by setting 
\be \l{vto} 
\gamma_{f,g}(c_{i,j}) = c_{i',j'}, \quad \text{where } a_{i'} = \gamma_f(a_i), \ \ b_{j'} = \gamma_g(b_j).
\ee 

\bt \l{t5} Let $f,p,g,q$ be branched covering maps between closed surfaces such that the diagram 
\be 
\begin{CD}
R @>q>> C_2\\
@VV p V @VV g V\\ 
C_1 @>f >> C 
\end{CD}
\ee
commutes and $h = f \circ p = g \circ q$. Assume that $p$ and $q$ have no non-trivial common compositional right factor and $\deg p = \deg g$. Then for all $z_0 \in C \setminus S_h$ and \linebreak $\gamma \in \pi_1(C \setminus S_h, z_0)$ the equality $\gamma_h = \gamma_{f,g}$ holds. 
\et

\pr Since 
\[
p^{-1}\{a_i\},\ 1 \leq i \leq n, \quad \text{and} \quad q^{-1}\{b_j\},\ 1 \leq j \leq m,
\]
are blocks with respect to the action of $G_h$, the image of $c_{i,j}$ under $\gamma_h$ lies in the intersection 
\[
p^{-1}\{\gamma_f(a_i)\} \cap q^{-1}\{\gamma_g(b_j)\},
\]
which, in the above notation, consists of the single point $c_{i',j'}$.
\qed

When performing calculations, it is convenient to treat \(c_{i, j}\), \(1 \leq i \leq n\), \linebreak \(1 \leq j \leq m\), as the entries of an \(n \times m\) matrix \(M\). Then, by Theorem \ref{t5}, the action of the permutation \(\gamma_{h}\) corresponds to permuting the rows of \(M\) according to \(\gamma_{f}\) and permuting the columns of \(M\) according to \(\gamma_{g}\).

\vskip 0.2cm
\noindent{\it Proof of Theorem \ref{t1}.}
For $h=f\circ p=g\circ q$, let $z'$ be a point in
$C_1\setminus f^{-1}(S_h)$ sufficiently close to $z$, and let
$z_0=f(z')$. Then, the partition $\mu_{p,z}$ coincides with the
collection of cycle lengths of the permutation $\sigma$ of the set
$p^{-1}\{z'\}$ arising from the lift of a small loop
$\gamma'\subset C_1$ around $z$. To prove the equality
\be \label{eqq}
\mu_{p,z}=\mu_{g,f(z)}^{\deg_z f},
\ee
we consider the permutation $\sigma$ as a restriction of the
permutation $\gamma_h$ of the set $h^{-1}\{z_0\}$ to
$p^{-1}\{z'\}$ for a suitable element
$\gamma\in\pi_1(C\setminus S_h,z_0)$, and then apply
Theorem \ref{t5} to $\gamma_h$ to relate $\mu_{p,z}$ to 
$\mu_{g,f(z)}$.

Specifically, consider a small loop
$\gamma\in\pi_1(C\setminus S_h,z_0)$ that wraps $\deg_z f$ times
around the point $f(z)$. Then, the preimage
$f^{-1}(\gamma)\subset C_1$ contains a loop around $z$ homotopic
to $\gamma'$, which implies that $\sigma$ is a restriction of the
permutation $\gamma_h$ of $h^{-1}\{z_0\}$ to $p^{-1}\{z'\}$.
Since $\gamma_h=\gamma_{f,g}$ by Theorem \ref{t5} and $\gamma_f$ fixes $z'$, $\gamma_h$ 
fixes the block $p^{-1}\{z'\}$, and formulas \eqref{per} and
\eqref{vto} imply that the collection of cycle lengths of $\sigma$
coincides with the collection of cycle lengths of the permutation
$\gamma_g$ of the set $g^{-1}\{z_0\}$. On the other hand, by the
construction of $\gamma$, the permutation $\gamma_g$ is the
$\deg_z f$-th power of the permutation $\t\sigma$ of the set
$g^{-1}\{f(z')\}$ arising from the lift of a small loop
$\t\gamma\subset C$ around $f(z)$. Therefore, \eqref{eqq} holds.

Let us observe now that if
$$
\mu=(a_1,a_2,\dots,a_k)
\qquad {\rm and}\qquad
\mu^d=(b_1,b_2,\dots,b_l)
$$
are partitions of $n\geq 1$ for some $d\geq 1$, then
\be \la{zx}
\lcm(a_1,a_2,\dots,a_k)
=
\lcm(b_1,b_2,\dots,b_l)\cdot
\gcd\bigl(d,\lcm(a_1,a_2,\dots,a_k)\bigr).
\ee
To prove \eqref{zx}, fix a prime number $p$ and set
$$
e_i=v_p(a_i),\qquad 1\leq i\leq k,
\qquad
e=v_p(d),
\qquad
M=\max(e_1,e_2,\dots,e_k).
$$
By the definition of $\mu^d$, each number $a_i$ is replaced by
the number
$$
\frac{\lcm(a_i,d)}{d}
=
\frac{a_i}{\gcd(a_i,d)}
$$
taken $\gcd(a_i,d)$ times. The $p$-adic valuation of this number is
$$
v_p\left(\frac{a_i}{\gcd(a_i,d)}\right)
=
e_i-\min(e_i,e)
=
\max(0,e_i-e).
$$
Since repeating a number does not change the least common
multiple, we have
$$
v_p\bigl(\lcm(b_1,b_2,\dots,b_l)\bigr)
=
\max_{1\leq i\leq k}\max(0,e_i-e)
=
\max(0,M-e).
$$
Moreover,
$$
v_p\bigl(\gcd(d,\lcm(a_1,a_2,\dots,a_k))\bigr)
=
\min(e,M).
$$
Consequently, the $p$-adic valuation of the right-hand side of
\eqref{zx} is
$$
\max(0,M-e)+\min(e,M)=M, 
$$
which is equal to
$$
v_p\bigl(\lcm(a_1,a_2,\dots,a_k)\bigr).
$$
Since this holds for every prime number $p$, equality
\eqref{zx} follows.

Applying \eqref{zx} to $\mu=\mu_{g,f(z)}$ and $d=\deg_z f$, and using equality \eqref{eqq} alongside the definitions of the orbifolds $\mathcal O_2^p$ and $\mathcal O_2^g$, we obtain the equality
$$
\nu_2^g(f(z)) = \nu_2^p(z) \cdot \gcd\bigl(\deg_z f, \nu_2^g(f(z))\bigr),
$$
which means that $f \colon \mathcal O_2^p \to \mathcal O_2^g$ is a minimal holomorphic map between orbifolds.

Finally, to prove that
$q:{\f O_1^p}\rightarrow{\f O_1^g}$ is also a minimal
holomorphic map between orbifolds, observe that since
$f:{\f O_2^p}\rightarrow{\f O_2^g}$ is a minimal holomorphic map
and
$p:{\f O_1^p}\rightarrow{\f O_2^p}$ is a covering map, the
composition
$f\circ p:{\f O_1^p}\rightarrow{\f O_2^g}$ is a minimal
holomorphic map by Corollary \ref{serka0}. It follows now from
the equality
\be \l{itfo}
f(p(z))=g(q(z))
\ee
by Corollary \ref{indu2} that
$q:{\f O_1^p}\rightarrow g^*({\f O_2^g})$ is a minimal
holomorphic map. Since
$g^*({\f O_2^g})={\f O_1^g}$, this implies the statement.
\qed

Let us mention that Theorem \ref{t1} and Theorem \ref{t5} adapt results from the holomorphic setting, where the extra assumption $C=\C\P^1$ was imposed (see \cite{pak}, Proposition 2.1, and \cite{semi}, Lemma 2.2 and Theorem 4.2).

\section{Solutions of 
\( f \circ p = p \circ q \)}

\subsection{\l{s31} Reduction to good solutions} 
We start by showing that the problem of describing branched covering maps  that make the diagram 
\be \l{dia2} 
\begin{CD}
R @>q>> R\\
@VV p V @VV p V\\ 
S^2 @>f >> S^2
\end{CD}
\ee
commutative reduces to the case where $p$ and $q$ have no non-trivial common compositional right factor, and hence to the setting where Theorem \ref{t1} is applicable. As in the holomorphic context, we will call such solutions of \eqref{se} {\it primitive}.

We will always assume that $\deg f \geq 2$. Since $\deg q = \deg f$, we also have $\deg q \geq 2$, which implies by the Riemann--Hurwitz formula that either $R = S^2$  or $R = T^2$.

The following result is a simplified version of Theorem 3.2 in \cite{lattes}, which was established for holomorphic maps.

\bt \l{los}
Let $f,p,q$ be branched covering maps  between closed surfaces such that $\deg f\geq 2$ and 
the diagram 
\be 
\begin{CD}
R @>q>> R\\
@VV p V @VV p V\\ 
S^2 @>f >> S^2
\end{CD}
\ee
commutes. Then there exist
branched covering maps    
$\psi:R\rightarrow R,$ $p_0:R\rightarrow S^2$, $q_0: R\rightarrow R$
satisfying the following conditions.

\begin{enumerate}[label=\upshape(\alph*)]
\item 
The diagram 
\be 
\begin{CD} \l{xxuu}
R @>q>> R \\
@VV  \psi V @VV  \psi  V\\ 
R @> q_0 >> R\\
@VV p_0 V @VV p_0 V\\ 
S^2 @>f >> S^2  
\end{CD}
\ee
commutes, and $p=p_0\circ \psi$.

\item The relation $q_0\sim q$ holds.

\item The maps $p_0$, $q_0$,  and $f$  form a primitive solution of $f\circ p=p\circ q$. Moreover, $\deg p_0\geq 2$, unless $q\sim f$ and \( R=S^2 \).
\end{enumerate}
\et

\pr 
If $p$ and $q$ have no non-trivial common compositional right factor, then we can set  
$$q_0=q, \quad p_0=p, \quad \psi=\mathrm{id}.$$ 
Otherwise, there exist a closed surface $R^{\prime}$ and branched covering maps 
$$u_1:R\rightarrow R^{\prime}, \quad p^{\prime}:R^{\prime}\rightarrow S^2, \quad v_1:R^{\prime}\rightarrow R,$$
such that   
\be \l{eli} p=p^{\prime}\circ u_1, \quad q=v_1\circ u_1,\ee
and $\deg u_1 \geq 2.$ Furthermore, since $q:R\rightarrow R$ decomposes as 
$$R\overset{u_1}{\longrightarrow}R^{\prime}\overset{v_1}{\longrightarrow}R,$$
the equality $g(R^{\prime})=g(R)$ holds. So, we can assume that $R'=R$. 

By substituting \eqref{eli} into the equation \( f \circ p = p \circ q \), we see that 
\[
f \circ p' = p' \circ u_1 \circ v_1,
\]
and the diagram 
\be \l{dai} 
\begin{CD}
R @>q=v_1 \circ u_1>> R \\
@VV  u_1 V @VV  u_1  V\\
R^{\prime} @> q'=u_1 \circ v_1 >> R^{\prime}\\
@VV p^{\prime} V @VV p^{\prime} V\\
S^2 @>f >> S^2
\end{CD}
\ee 
commutes. 
If the maps \( p' \) and \(q'= u_1 \circ v_1 \) still have a non-trivial common right factor, we can apply a similar transformation again. Since \( \deg u_1 \geq 2 \) implies \( \deg p' < \deg p \), after a finite number of steps we must reach the diagram \eqref{xxuu}, where \( q_0 \) is obtained from \( q \) through a sequence of elementary transformations, and the maps \( p_0 \) and \( q_0 \) no longer have a non-trivial common compositional right factor. 

Finally, if $\deg p_0 = 1$, then $f$ is conjugate to $q_0$ by an orientation-preserving  homeomorphism, which implies that $f \sim q$ and $R=S^2$.
\qed

\subsection{Proofs of Theorems \ref{t2} and \ref{t3}}
We start by proving the following lemma.

\begin{lemma}\l{ai} 
Let $\mathcal O=(S^2,\nu)$ be an orbifold of positive Euler
characteristic, and let $p:S^2\to S^2$ be a branched covering
map such that $p:\mathcal O\to\mathcal O$ is a minimal
holomorphic map. Then, for any decomposition $p=u\circ v$
into branched covering maps $u,v:S^2\to S^2$, the orbifolds
$u^*\mathcal O$ and $\mathcal O$ have the same signature.
\end{lemma}

\begin{proof}
In the holomorphic setting the lemma was proved in
\cite[Corollary~5.1]{semi} and the general case can be reduced to the holomorphic case as follows. 
Choose an orientation-preserving homeomorphism
$h_0:S^2\to\mathbb{CP}^1$. Applying
\cite[Corollary~A.12]{bm} first to $u$ and then
to $v$, we obtain orientation-preserving homeomorphisms
$h_1,h_2:S^2\to\mathbb{CP}^1$ and rational
functions $U,V$ such that the following diagram commutes:
\[
\begin{CD}
S^2 @>v>> S^2 @>u>> S^2\\
@VV h_2 V @VV h_1 V @VV h_0 V\\
\mathbb{CP}^1 @>V>> \mathbb{CP}^1 @>U>> \mathbb{CP}^1
\end{CD}
\]
Moreover, since $\chi(\mathcal O)>0$, the set $c(\mathcal O)$
contains at most three points. Hence there exists a
M\"obius transformation $M$ such that
\[
M(h_2(z))=h_0(z),\qquad z\in c(\mathcal O).
\]
Replacing $h_2$ by $M\circ h_2$ and $V$ by $V\circ M^{-1}$,
we may therefore assume that $h_2$ and $h_0$ coincide on
$c(\mathcal O)$.

Set
\[
\mathcal O'=(\mathbb{CP}^1,\nu\circ h_0^{-1}).
\]
Since $h_2$ and $h_0$ coincide on $c(\mathcal O)$, we have
\[
\nu\circ h_2^{-1}=\nu\circ h_0^{-1}.
\]
Therefore,
\[
U\circ V=h_0\circ p\circ h_2^{-1}
\]
is a minimal holomorphic self-map of $\mathcal O'$.
Applying the holomorphic version of the lemma, we conclude
that $U^*\mathcal O'$ and $\mathcal O'$ have the same signature.

Furthermore, the commutativity of the diagram gives
\[
h_1^*(U^*\mathcal O')
=u^*(h_0^*\mathcal O')
=u^*\mathcal O.
\]
Thus, since $h_1$ is a homeomorphism, $u^*\mathcal O$ has the same
signature as $U^*\mathcal O'$, and hence as $\mathcal O'$.
Finally, $\mathcal O'$ and $\mathcal O$ have the same signature
because $h_0$ is a homeomorphism. This proves the lemma.
\end{proof}

\noindent{\it Proof of Theorem \ref{t3}.}
Assuming that $f\not\sim q$, consider the lower square in
diagram \eqref{xxuu} from Theorem \ref{los}. Applying
Theorem \ref{t1}, we see that
$q_0:\f O_1^{p_0}\rightarrow \f O_1^{p_0}$ and
$f:\f O_2^{p_0}\rightarrow \f O_2^{p_0}$ are minimal
holomorphic maps between orbifolds, implying by
\eqref{ner2} that $\chi(\f O_1^{p_0})\geq 0$ and
$\chi(\f O_2^{p_0})\geq 0$. Moreover, since $f\not\sim q$
implies $\deg p_0\geq 2$, the map $p_0$ has non-trivial
branching, implying that $\f O_2^{p_0}$ is distinct from
$\f S^2$. Since
$p_0:\f O_1^{p_0}\rightarrow \f O_2^{p_0}$ is a covering
map between orbifolds, we also have
$$
\chi(\f O_1^{p_0})=\deg p_0\,\chi(\f O_2^{p_0}).
$$
Thus, either both Euler characteristics are zero, or
$$
0<\chi(\f O_2^{p_0})<\chi(\f O_1^{p_0}).
$$

In case $\deg \psi=1$, we may assume that 
$p_0=p$,  and setting $\f O_1=\f O_1^{p_0}$
and $\f O_2=\f O_2^{p_0}$, we see that 
$$
q:\f O_1\rightarrow \f O_1,\qquad
f:\f O_2\rightarrow \f O_2,\qquad
p:\f O_1\rightarrow \f O_2
$$
are minimal holomorphic maps by Theorem \ref{t1}. 
Moreover, since both orbifolds are good, if $\chi(\f O_2)>0$, then 
the collections of ramification indices of $\f O_1$ and
$\f O_2$ have the required form by the well-known
classification of orbifolds of non-negative Euler
characteristic on $S^2$.

In the general case, still setting $\f O_2=\f O_2^{p_0}$,
we apply Corollaries \ref{indu2} and \ref{serka0} to the
steps from Theorem \ref{los} in reverse order.
Specifically, using induction, it is enough to prove
the following statement for a single step described by
the diagram \eqref{dai}: if $\f O$ and $\f O_2$ are
orbifolds of non-negative Euler characteristic such that
$$
f:\f O_2\rightarrow \f O_2,\qquad
p':\f O\rightarrow \f O_2,\qquad
u_1\circ v_1:\f O\rightarrow \f O
$$
are minimal holomorphic maps between orbifolds, then
there exists an orbifold $\f O_1$ with
$\chi(\f O_1)=\chi(\f O)$ such that
\be \l{sum}
p'\circ u_1:\f O_1\rightarrow \f O_2
\quad {\rm and}\quad
v_1\circ u_1:\f O_1\rightarrow \f O_1
\ee
are minimal holomorphic maps between orbifolds.

Let us set $\f O_1=u_1^*(\f O)$. Then by
Corollary \ref{indu2},
$$
u_1:\f O_1\rightarrow \f O
\quad {\rm and}\quad
v_1:\f O\rightarrow \f O_1
$$
are minimal holomorphic maps, which implies by
Corollary \ref{serka0} that the maps \eqref{sum} are also
minimal holomorphic maps.

If $\chi(\f O)>0$, applying Lemma \ref{ai} to
$u_1\circ v_1:\f O\rightarrow \f O$, we conclude that \linebreak 
$\f O_1=u_1^*(\f O)$ and $\f O$ have the same signature
and hence the same Euler characteristic. If
$\chi(\f O)=0$, applying \eqref{ner2} to $u_1$ and $v_1$
gives
$$
\chi(\f O_1)\leq \deg u_1\,\chi(\f O)=0,
\qquad
0=\chi(\f O)\leq \deg v_1\,\chi(\f O_1).
$$
Therefore, $\chi(\f O_1)=0$. 
This finishes the proof of Theorem \ref{t3}. \qed

\vskip 0.2cm
\noindent{\it Proof of 
Theorem \ref{t2}.}
In this case, the condition $f\not\sim q$ is satisfied automatically, as $T^2$ and $S^2$ have different genera. Considering the lower square in the diagram \eqref{xxuu} from Theorem \ref{los}, we observe that since $\chi(\f T^2)=0$, the inequality $\chi(\f O_1^{p_0})\geq 0$ implies by formula \eqref{euler} that $\chi(\f O_1^{p_0})= 0$ and $\f O_1^{p_0}=\f T^2$. Since $p_0:\f O_1^{p_0}\rightarrow \f O_2^{p_0}$ is a covering map, this implies that $\chi(\f O_2^{p_0})= 0$ as well by \eqref{ner1}. Moreover, since $f:\f O_2^{p_0}\rightarrow \f O_2^{p_0}$ and $q_0:\f T^2\rightarrow \f T^2$ are minimal holomorphic maps for which equality in inequality \eqref{ner2} is attained, they are covering maps. 
Now the proof can be finished using Corollaries \ref{indu2} and \ref{serka0} in the same way as in the proof of Theorem \ref{t3}. \qed 
\vskip 0.2cm
For analogues of Theorems \ref{t2} and \ref{t3} in the holomorphic setting, we refer the reader to \cite{semi} (Theorem 6.1), \cite{dyna} (Theorem 1.1), and \cite{lattes} (Theorem 1.1).

\subsection{\l{s32} Quotients of sphere and torus endomorphisms} 
Let $f\colon S^2 \to S^2$ be a branched covering map. Recall that $f$ is called \emph{postcritically finite}, or a \emph{Thurston map}, if the forward orbit of every critical point of $f$ under iteration is finite. For a Thurston map $f\colon S^2 \to S^2$, we define its orbifold $\f O_f$ by setting $\nu(z)$ to be the least common multiple of the local degrees $\deg_x f^n$ taken over all $n \geq 1$ and all $x \in S^2$ satisfying $f^n(x) = z$. Note that these local degrees can be unbounded. This occurs if and only if $z$ belongs to a periodic cycle containing a critical point, in which case we set $\nu(z) = \infty$  (here we allow infinite ramification indices).

In this notation, the question from \cite{bm}, \cite{bm2} mentioned in the introduction asks whether every quotient of a torus endomorphism---that is, a branched covering map $f$ of degree at least two for which there exist branched covering maps $p$ and $q$ such that the diagram \eqref{d} commutes---is a Thurston map without periodic critical points satisfying $\chi(\f O_f) = 0$. It was already shown in \cite[Lemma 3.12]{bm} that such a map $f$ is indeed a Thurston map without periodic critical points. Furthermore, it was also established in \cite[Lemma 3.13]{bm} that an affirmative answer to the above question would follow if the map $p$ has the same local degree at every point of $p^{-1}\{z\}$ for each $z \in S^2$. 
On the other hand, it is easy to see that the last property of $p$ is a corollary of Theorem \ref{t2}, which thus answers the question in the affirmative. Indeed, since $p \colon \f T^2 \to \f O$ is a covering map and $\f T^2$ is unramified, the local degree at every point of $p^{-1}\{z\}$ is equal to $\nu(z)$.

Note that any covering map between orbifolds $f\colon \f O \to \f O$ is a Thurston map. Indeed, \eqref{us} implies that $f$ maps any critical point of $f$ to the set $c(\f O)$. On the other hand, \eqref{us} implies that $$f(c(\f O))\subseteq c(\f O).$$   
A minimal holomorphic map $f\colon \f O \to \f O$ is in general not a Thurston map. Nevertheless, condition \eqref{rys} still implies that $f$ maps the set $c(\f O)$ into itself. Thus, $f$ is a Thurston map whenever $S_f$ is a subset of $c(\f O)$. 

We recall that in the holomorphic setting, maps from a compact Riemann surface to $\C\P^1$  unramified outside $\{0,1,\infty\}$ are called
\emph{Belyi maps} and constitute the subject of a deep theory
known as ``dessins d'enfants'', which goes back to Grothendieck
(see, e.g., \cite{des1}). It is intriguing to explore whether a Belyi map $f:\C\P^1\rightarrow \C\P^1$ that is also a minimal holomorphic map $f\colon \f O \to \f O$ for some orbifold $\f O\neq \f S^2$ possesses any remarkable arithmetical or dynamical properties.

\section{Arbitrary solutions of the equation 
\( f \circ p = g \circ q \)}
\subsection{Fiber products} 

In this section, we prove a number of results about solutions of \eqref{fu} without additional assumptions. A natural approach here is to fix $f$ and $g$ and describe all possible $h$ satisfying 
\begin{equation}\label{ii} 
h = f \circ p = g \circ q,
\end{equation}
by utilizing the fiber product construction.

For branched covering maps $f$ and $g$ with the same target, we denote by $S_{f,g}$ the union $S_f \cup S_g$.

\bt \la{p1} Let $f:\, C_1\rightarrow C$ and $g:\, C_2\rightarrow C$ be 
branched covering maps between closed surfaces, and $z_0$ a point in $C\setminus S_{f,g}$.
Then for any $a\in f^{-1}\{z_0\}$ and 
$b\in g^{-1}\{z_0\}$, there exist 
a closed surface $R$, 
branched covering maps $p:\, R\rightarrow C_1,$ $q:\, R\rightarrow C_2$,
$h:\, R\rightarrow C,$ and a point $c\in h^{-1}\{z_0\}$ 
such that:

\begin{enumerate}[label=\upshape(\alph*)]
\item The equalities 
\be \la{f1} h=f\circ p=g\circ q, \ \ \ p(c)=a, \ \ \ q(c)=b\ee hold.
\item The maps $p$ and $q$ have no non-trivial common compositional right factor. 
\item The inequalities $$\deg q\leq \deg f, \qquad \deg p\leq \deg g$$ hold. 
\end{enumerate} 
Furthermore, for any branched covering maps $\t p:\, \t R\rightarrow C_1,$ $\t q:\, \t R\rightarrow C_2$,  $\t{h}  :\, \t R\rightarrow C,$ 
and a point $\t c\in \t{h}  ^{-1}\{z_0\}$ satisfying 
\be \la{f2} \t{h}  =f\circ \t p=g\circ \t q, \ \ \  \t p(\t c)=a, 
\ \ \ \t q(\t c)=b, \ee 
there exists a unique branched covering map $w:\, \t R\rightarrow R$ such that
\be \la{e3} \t{h}  = h\circ w, \ \ \ \t p= p\circ  w,\ \ \ \t q= q\circ w,
\ \ \ w(\t c)=c.\ee
\et

\pr Since the subgroups $\Gamma_{f,a}$ and $\Gamma_{g,b}$ are of finite index in 
$\pi_1(C\setminus S_{f,g},z_0)$, their intersection is also of finite index. Therefore, there exists 
a pair $h:\, R\rightarrow C,$ $c\in h^{-1}\{z_0\}$ such that 
\be \l{mor} \Gamma_{h,c}=\Gamma_{f,a}\cap\Gamma_{g,b},\ee 
and for such a pair condition $(a)$ holds. 
Moreover,  by Proposition  \ref{l1}, the condition \eqref{mor} implies $(b)$, while $(c)$ follows from 
\be 
\begin{split}
\deg h & = [\pi_1(C\setminus S_{f,g},z_0):\Gamma_{h,c}] =[\pi_1(C\setminus S_{f,g},z_0):\Gamma_{f,a}\cap\Gamma_{g,b}]\\
& \leq [\pi_1(C\setminus S_{f,g},z_0):\Gamma_{f,a}]\cdot [\pi_1(C\setminus S_{f,g},z_0):\Gamma_{g,b}] \\
& =\deg f \cdot \deg g.
\end{split}
\ee

Finally, assume that \eqref{f2} holds. First suppose that
$z_0\notin S_{\t h}$, and regard all covering subgroups
as subgroups of
\[
\pi_1(C\setminus(S_{f,g}\cup S_{\t h}),z_0).
\]
Adding punctures preserves the equality
\[
\Gamma_{h,c}=\Gamma_{f,a}\cap\Gamma_{g,b}.
\]
Since
\[
\Gamma_{\t h,\t c}
\subseteq\Gamma_{f,a}\cap\Gamma_{g,b}=\Gamma_{h,c},
\]
the correspondence between coverings and subgroups gives
a branched covering map $w:\t R\to R$ such that
\[
\t h=h\circ w,\qquad w(\t c)=c.
\]

It then follows from
\[
f\circ\t p=f\circ p\circ w,\qquad
g\circ\t q=g\circ q\circ w
\]
and
\[
(p\circ w)(\t c)=\t p(\t c),\qquad
(q\circ w)(\t c)=\t q(\t c)
\]
that equalities \eqref{e3} hold. 

If $z_0$ is a critical value of $\t h$, choose a path
$\lambda$ in $\t R\setminus\t h^{-1}(S_{f,g})$
from $\t c$ to a point $\t c_1$ such that
$z_1=\t h(\t c_1)$ is not a critical value of $\t h$.
Let $\lambda_0$ be the lift of $\t h\circ\lambda$
under $h$ starting at $c$, and let $c_1$ be its endpoint.
Composing $\lambda_0$ with $p$ and $q$ gives
\[
p(c_1)=\t p(\t c_1),\qquad
q(c_1)=\t q(\t c_1).
\]
Changing the basepoint along these paths gives
\[
\Gamma_{h,c_1}
=\Gamma_{f,p(c_1)}\cap\Gamma_{g,q(c_1)}.
\]
Thus, the preceding argument applied at $z_1$ gives
a branched covering map \linebreak $w:\t R\to R$ such that
\[
\t h=h\circ w,\qquad w(\t c_1)=c_1.
\]

Since $\t h=h\circ w$, we have
\[
h\circ(w\circ\lambda)=\t h\circ\lambda=h\circ\lambda_0.
\]
Thus, $w\circ\lambda$ and $\lambda_0$ are lifts of
the same path under $h$, and both end at $c_1$.
By uniqueness of path lifting applied in the reverse
direction, $w\circ\lambda=\lambda_0$.
Comparing their starting points, we obtain that the equality $w(\t c)=c$ still holds. \qed

The maps $h,p,q$ constructed above depend in general on
the choice of the pair $(a,b)\in
f^{-1}\{z_0\}\times g^{-1}\{z_0\}$.
To describe all possible $h$ together with the maps
$p$ and $q$, we consider the following diagonal
monodromy action of $\pi_1(C\setminus S_{f,g},z_0)$.
This modifies the action introduced before the proof
of Theorem \ref{t5}. The group now acts on
$f^{-1}\{z_0\}\times g^{-1}\{z_0\}$ instead of
$h^{-1}\{z_0\}$, and the action need not be transitive.

Specifically, for $\gamma\in\pi_1(C\setminus S_{f,g},z_0)$,
denoting as above by $\gamma_f$ and $\gamma_g$ the
corresponding permutations of $f^{-1}\{z_0\}$
and $g^{-1}\{z_0\}$, we set
\[
\gamma_{f,g}\cdot(a,b)
=(\gamma_f(a),\gamma_g(b)).
\]
\bt \la{p11} Let $f:\, C_1\rightarrow C$ and $g:\, C_2\rightarrow C$ be 
branched covering maps  between closed surfaces. Then there exists a collection
\be \l{nota} (C_1,f)\times(C_2,g)
=
\bigl\{(R_j,p_j,q_j,h_j)\bigr\}_{j=1}^{o(f,g)},\ee 
where $o(f,g)$ is a  positive integer and $R_j$ are closed surfaces provided with branched covering  maps
$$p_j:\, R_j\rightarrow C_1, \ \ \ q_j:\, R_j\rightarrow C_2, \ \ \ 
h_j:\, R_j\rightarrow C,\ \ \ 1\leq j \leq o(f,g),$$ such that: 
\begin{enumerate}[label=\upshape(\alph*)]
\item The equalities 
\be \la{pes} h_j=f\circ p_j=g\circ q_j, \ \ \ 1\leq j \leq o(f,g),\ee hold.
\item The maps $p_j$ and $q_j$, $  1\leq j \leq o(f,g)$, 
have no non-trivial common compositional right factor. 
\item The equalities    
\be \sum_{j}\deg p_j=  \deg g,\ \ \ \ \sum_{j}\deg q_j= \deg f\ee hold. 
\end{enumerate} 

Furthermore, for any branched covering  maps $p:\, R\rightarrow C_1,$  $q:\, R\rightarrow C_2$ and $h:\, R\rightarrow C$ 
 satisfying 
\be \l{mm} h=f\circ p=g\circ q\ee there exist a uniquely defined  index $j$ and 
a branched covering  map $w:\, R\rightarrow R_j$ such that the equalities 
\be \l{univ}  h= h_j\circ w, \ \ \ p= p_j\circ  w, \ \ \ q= q_j\circ w\ee
hold. 
\et
\pr
Let $\mathcal O_1,\dots,\mathcal O_{o(f,g)}$ be the orbits
of the diagonal action of $\pi_1(C\setminus S_{f,g},z_0)$ on
$f^{-1}\{z_0\}\times g^{-1}\{z_0\}$, for a fixed $z_0\in C\setminus S_{f,g}$. 
Choose a representative $(a_j,b_j)$ in each orbit.
Applying Theorem \ref{p1} to $(a_j,b_j)$, we obtain
$R_j,p_j,q_j,h_j$ satisfying $(a)$ and $(b)$, together
with a point $c_j\in h_j^{-1}\{z_0\}$ such that
\[
p_j(c_j)=a_j,\qquad q_j(c_j)=b_j.
\]

Assume now that \eqref{mm} holds. Observe that
one of the chosen representatives $(a_j,b_j)$ has
the form $(p(c),q(c))$ for a point
$c\in h^{-1}\{z_0\}$.
Indeed, choose any $d\in h^{-1}\{z_0\}$, and let
$\mathcal O_j$ be the orbit containing $(p(d),q(d))$.
By the definition of the orbit, there exists
$\gamma\in\pi_1(C\setminus S_{f,g},z_0)$ such that
\[
\gamma_f(p(d))=a_j,\qquad \gamma_g(q(d))=b_j.
\]
Lift a representative of $\gamma$ under $h$ starting
at $d$ (not necessarily uniquely), and let $c$ be its endpoint. Composing this
lift with $p$ and $q$ gives lifts under $f$ and $g$,
respectively, so
\[
p(c)=\gamma_f(p(d))=a_j,\qquad
q(c)=\gamma_g(q(d))=b_j.
\]
The universal property of $h_j,p_j,q_j$ now gives
a branched covering map \linebreak $w:R\to R_j$ satisfying
\eqref{univ} and $w(c)=c_j$.

To prove the uniqueness of $j$, suppose that for some
$i$ there exists a branched covering map
$\t w:R\to R_i$ such that
\[
h=h_i\circ\t w,\qquad
p=p_i\circ\t w,\qquad q=q_i\circ\t w.
\]
Since $\t w$ is surjective, there exists $e\in R$ with
$\t w(e)=c_i$. Then
\[
h(e)=z_0,\qquad p(e)=a_i,\qquad q(e)=b_i.
\]
Join $c$ to $e$ by a path in $R\setminus h^{-1}(S_{f,g})$,
and let $\delta\in\pi_1(C\setminus S_{f,g},z_0)$
be represented by its image under $h$. Then
\[
\delta_f(a_j)=a_i,\qquad \delta_g(b_j)=b_i.
\]
Thus, $(a_j,b_j)$ and $(a_i,b_i)$ belong to the
same orbit, so $i=j$.

Finally, the stabilizer of $(a_j,b_j)$ is
$\Gamma_{f,a_j}\cap\Gamma_{g,b_j}$, which is
$\Gamma_{h_j,c_j}$ by construction. Consequently,
\[
|\mathcal O_j|
=[\pi_1(C\setminus S_{f,g},z_0):\Gamma_{h_j,c_j}]
=\deg h_j.
\]
Summing over the orbits gives
\[
\sum_{j=1}^{o(f,g)}\deg h_j
=\sum_j|\mathcal O_j|
=\deg f\cdot\deg g.
\]
Since
\[
\deg h_j=\deg f\,\deg p_j=\deg g\,\deg q_j,
\]
this proves $(c)$. \qed

We call the collection 
\be 
(C_1,f)\times(C_2,g)
=
\bigl\{(R_j,p_j,q_j,h_j)\bigr\}_{j=1}^{o(f,g)},
\ee 
which appears in Theorem \ref{p11}, the {\it fiber product} of $f$ and $g$, and we refer to the quadruples $(R_j,p_j,q_j,h_j)$ as its components. 
Note that if $f$ and $g$ are holomorphic maps between compact Riemann surfaces, then these components are in a one-to-one correspondence with the components of the algebraic curve \be \l{cura} f(x)=g(y)\ee defined in the product of $C_1$ and $C_2$. We will say that the fiber product of $f$ and $g$ is irreducible if $o(f,g)=1$ and reducible otherwise.

Theorem \ref{p11} easily implies the following useful criterion, proved in the holomorphic setting in \cite{semi}, which allows one to apply Proposition \ref{l2} to the study of irreducible fiber products.

\begin{corollary} \l{21} A solution $f,g,p,q$ of the equation $f\circ p=g\circ q$
in branched covering maps between closed surfaces is good
whenever any two of the following three conditions are satisfied:
\begin{itemize}
    \item The fiber product of $f$ and $g$ is irreducible,
    \item The maps $p$ and $q$ have no non-trivial common compositional right factor,
    \item The equalities $\deg f = \deg q$ and $\deg g = \deg p$ hold. \qed
\end{itemize}
\end{corollary}

Fix a labeling $\{z_1, z_2, \dots , z_r\}$ of the points of $S_{f,g}$, and let 
$$
\mu_{f,z_i}=(f_{i,1},f_{i,2}, \dots , f_{i,u_i}), \qquad 
\mu_{g,z_i}=(g_{i,1},g_{i,2}, \dots , g_{i,v_i}), \qquad 1\leq i \leq r. 
$$
The following result generalizes the corresponding result of Fried (see \cite{f3}, Proposition 2) concerning the genus of the irreducible algebraic curve \eqref{cura} defined by rational functions $f,g$ on the Riemann sphere.

\bc \la{p3} Let $f:\, C_1\rightarrow C$ and $g:\, C_2\rightarrow C$ be 
branched covering maps   between closed surfaces. Then, in the above notation,  the formula 
\be \la{rh0} 
\sum_{j=1}^{o(f,g)}\bigl(2-2g(R_j)\bigr) =
\sum_{i=1}^{r}
\sum_{j_1=1}^{u_{i}} \sum_{j_2=1}^{v_{i}} \gcd(f_{i,j_1}, g_{i,j_2}) \;-\; \bigl(r+2g(C)-2\bigr)\,\deg f\,\deg g
\ee
holds.
\ec
\pr
For each $i$, $1\leq i \leq r$, fix a small loop $\gamma^i\in \pi_1(C\setminus S_{f,g},z_0)$ around $z_i$ such that the partitions $\mu_{f,z_i}$ and $\mu_{g,z_i}$ coincide with the collections of cycle lengths in the permutations $\gamma^i_f$ and $\gamma^i_g$.  
For any $j$, $1\leq j \leq o(f,g)$, the Riemann–Hurwitz formula for the branched covering map $h_j:R_j\to C$ gives
\[
2-2g(R_j) = \deg h_j \cdot (2-2g(C)) - \sum_{i=1}^{r}\bigl(\deg h_j - e_i(j)\bigr),
\]
where $e_i(j)$ is the number of disjoint cycles in the permutation $\gamma^i_{h_j}$ acting on $h_j^{-1}(z_0)$.  By construction, $e_i(j)$ is also the number of cycles
of $\gamma^i_{f,g}$ on the orbit corresponding to $R_j$. 
Equivalently,
\[
2-2g(R_j) = \sum_{i=1}^{r} e_i(j) - \bigl(r+2g(C)-2\bigr)\deg h_j.
\]
Summing over all components $R_j$ and using that $$\sum_j \deg h_j = \deg f\,\deg g$$ by Theorem \ref{p11}, we obtain
\[
\sum_{j=1}^{o(f,g)}\bigl(2-2g(R_j)\bigr) = \sum_{j=1}^{o(f,g)}\sum_{i=1}^{r} e_i(j) \;-\; \bigl(r+2g(C)-2\bigr)\deg f\,\deg g.
\]

On the other hand, it follows from the construction of the diagonal monodromy action that a pair of cycles of lengths $d_1$ and $d_2$ in the permutations $\gamma^i_f$ and $\gamma^i_g$, $1\leq i \leq r$, gives rise to $\gcd(d_1,d_2)$ cycles of length $\operatorname{lcm}(d_1,d_2)$ in the permutation $\gamma^i_{f,g}$. 
Thus, for a given $i$,
\[
\sum_{j=1}^{o(f,g)} e_i(j) = \sum_{j_1=1}^{u_i}\sum_{j_2=1}^{v_i} \gcd(f_{i,j_1}, g_{i,j_2}).
\]
Substituting this into the previous sum yields the desired formula. \qed

\subsection{Normalizations and Fried's theorem}   
In general, it is difficult to determine when the fiber product of two 
branched covering maps has a single component. For example, even when 
$f$ and $g$ are polynomials, a complete description of reducible 
algebraic curves of the form \eqref{cura} --- the so-called 
Davenport--Lewis--Schinzel problem --- was the subject of numerous 
studies over several decades. Following a long series of partial 
results, the problem was recently settled in \cite{nef}.  
Nevertheless, a number of results concerning reducibility of curves 
of the form \eqref{cura} have a very general character, and remain true for arbitrary branched covering maps. In this 
final section we provide several such results. 

We start with the following statement. 

\bp \la{p65}
The fiber product of two branched covering maps $f: C_1\rightarrow C$ and $g: C_2\rightarrow C$   between closed surfaces is irreducible if and only if for any $z_0\in C\setminus S_{f, g}$ and any $a\in f^{-1}(z_0),$ $b\in g^{-1}(z_0),$ the equality 
\be \l{ieq} 
\Gamma_{f,a}\Gamma_{g,b} = \Gamma_{g,b}\Gamma_{f,a} = \pi_1(C\setminus S_{f, g}, z_0)
\ee 
holds. 
\ep

\pr Let $R,p,q,h$ be a component of the fiber product of $f$ and $g$. Then, by Corollary \ref{21}, the equality $o(f,g)=1$ holds if and only if  $f,g,p,q$ form a good solution of the equation $f\circ p=g\circ q$. Since, by properties of fiber products, $p$ and $q$ have no non-trivial common compositional right  factor, the  proposition follows now from Proposition  \ref{l2}. \qed

Proposition \ref{p65} implies the following corollary. 

\bc \l{corcor} Let  $f: C_1\rightarrow C$ and $g: C_2\rightarrow C$  be branched covering maps  between closed surfaces such that their fiber product is irreducible. Then for any decomposition of $f$ into a composition of  branched covering maps  between closed surfaces  $f=f_1\circ f_2$, the fiber product of $f_1$ and $g$ is also irreducible.  
\ec
\pr Indeed, it is easy to see that equality \eqref{ieq}  implies that $$
\Gamma\,\Gamma_{g,b} = \Gamma_{g,b}\Gamma = \pi_1(C\setminus S_{f, g}, z_0)$$
for any subgroup $\Gamma$ of $\pi_1(C\setminus S_{f, g}, z_0)$ containing 
$\Gamma_{f,a}$. \qed

The following  result generalizes the corresponding result about  polynomial curves \eqref{cura} (see \cite{ehr}, \cite{tve}).

\bp \la{p66} The fiber product of two branched covering maps $f: C_1\rightarrow C$ and $g: C_2\rightarrow C$   between closed surfaces is irreducible whenever their degrees are coprime.
\ep
\pr If the degrees of $f$ and $g$ are coprime, then part (a) of Theorem \ref{p11} implies that the degree of each $h_j$ is divisible by $\deg f \deg g$, implying that $\deg p_j \geq \deg g$ and $\deg q_j \geq \deg f$. Combined with part (c), this yields $o(f,g)=1$. \qed

Let $f: C_1\rightarrow C$ and $g: C_2\rightarrow C$ be branched covering maps. 
We say that $f$ and $g$ have a \emph{non-trivial common compositional left factor} if there exist a closed surface $\t C$ and branched covering maps 
$\tilde f: C_1\to \widetilde C$, $\tilde g: C_2\to \widetilde C$ and $w:\widetilde C\to C$ with $\deg w>1$ such that 
\be \label{fac} f = w\circ\tilde f,\qquad g = w\circ\tilde g. \ee

\bp \la{p67} 
Branched covering maps $f: C_1\rightarrow C$ and $g: C_2\rightarrow C$   between closed surfaces  have no non-trivial common  compositional left factor if and only if for any $z_0\in C\setminus S_{f,g}$ and any $a\in f^{-1}(z_0),$ $b\in g^{-1}(z_0),$ the equality 
\[
\langle \Gamma_{f,a},\Gamma_{g,b} \rangle = \pi_1(C\setminus S_{f,g}, z_0)
\] 
holds. 
\ep

\pr Assume that \eqref{fac} holds, and let $z_0\in C\setminus S_{f,g}$ and $\tilde{z}_0\in w^{-1}(z_0).$ Then for any 
$a\in \tilde{f}^{-1}(\tilde{z}_0)$ and $b\in \tilde{g}^{-1}(\tilde{z}_0),$ we have $\Gamma_{f,a}\subseteq\Gamma_{w,\tilde{z}_0}$ and $\Gamma_{g,b}\subseteq\Gamma_{w,\tilde{z}_0}$. Thus,
\[
\langle\Gamma_{f,a},\Gamma_{g,b}\rangle\subseteq\Gamma_{w,\tilde{z}_0},
\] 
and $\Gamma_{w,\tilde{z}_0}$ is a proper subgroup of $\pi_1(C\setminus S_{f,g}, z_0)$ because 
\[
[\pi_1(C\setminus S_{f,g}, z_0):\Gamma_{w,\tilde{z}_0}]=\deg w>1.
\] 

Conversely, if for some $a\in f^{-1}(z_0)$ and $b\in g^{-1}(z_0),$ the subgroup $\Gamma=\langle\Gamma_{f,a},\Gamma_{g,b}\rangle$ does not coincide with the entire group $\pi_1(C\setminus S_{f,g}, z_0)$, then, since it has finite index, it is of the form $\Gamma_{w,\tilde{z}_0}$ for some branched covering map $w$ of degree greater than 1, and the equalities \eqref{fac} hold for some $\tilde{f}$ and $\tilde{g}$. \qed

\bc \l{kaka}
If branched covering maps between closed surfaces $f$ and $g$
have a non-trivial common compositional left factor, then the
fiber product of $f$ and $g$ is reducible.
\ec

\pr
By Proposition \ref{p67},
$\langle\Gamma_{f,a},\Gamma_{g,b}\rangle$ is contained in the
proper subgroup $\Gamma_{w,\tilde z_0}$ of
$\pi_1(C\setminus S_{f,g},z_0)$. Hence,
\eqref{ieq} is impossible,  
and the fiber product of $f$ and $g$ is reducible by
Proposition \ref{p65}.
\qed

Note that the subgroups $\langle \Gamma_{f,a},\Gamma_{g,b} \rangle$, similar to the subgroups $\Gamma_{f,a}\cap \Gamma_{g,b}$, in general are different for different choices of 
$a\in f^{-1}(z_0)$ and $b\in g^{-1}(z_0)$. 
Note also that the proof of Proposition \ref{p67} shows that non-trivial common compositional left factors of $f$ and $g$ correspond to proper subgroups of 
$\pi_1(C\setminus S_{f,g},z_0)$ containing  subgroups of the form $\langle \Gamma_{f,a},\Gamma_{g,b} \rangle$.

For a surface $R$, denote by $\operatorname{Aut}(R)$ the group of
orientation-preserving homeomorphisms of $R$.
Let us recall that a branched covering map $h:R\rightarrow C$ is called a Galois covering if its deck transformation group
\[
\operatorname{Aut}(R,h)=\{\mu \in \operatorname{Aut}(R) \mid h\circ\mu = h\}
\]
acts transitively on the fibers of $h$. Equivalently, $h$ is a Galois covering if for any point $z_0\in C\setminus S_h$ and any $e\in h^{-1}(z_0)$, the subgroup $\Gamma_{h,e}$ is normal in $\pi_1(C\setminus S_{h},z_0).$ Thus, for a Galois covering $h$, all subgroups $\Gamma_{h,e}$, $e\in h^{-1}(z_0)$, are equal, and we will simply write $\Gamma_h$ for them.

The normalization of a branched covering map $h\colon R\to C$ is the Galois covering $\widehat{h}\colon \widehat{R}\to C$ satisfying $\widehat{h}=h\circ t$ for some branched covering map $t\colon \widehat{R}\to R$, with the following universal property: for every Galois covering $h'\colon R'\to C$ such that  $h'=h\circ t'$ for some branched covering map $t':R'\rightarrow R$, there exists a branched covering map $u\colon R'\to \widehat{R}$ such that $h'=\widehat{h}\circ u$.

\bt For any branched covering map $h:R\rightarrow C$ between closed surfaces, the normalization $\widehat{h} :\widehat{R}\rightarrow C$ exists and is unique up to equivalence.  Furthermore, $S_h=S_{\widehat h}$. 
\et
\pr
For $z_0\in C\setminus S_h$, let us consider the group $\bigcap_{e\in h^{-1}(z_0)} \Gamma_{h,e}$. This is a normal subgroup of finite index in $\pi_1(C\setminus S_h, z_0)$, hence it corresponds to a Galois covering $\widehat h:\widehat R\to C$ 
with 
\be \l{inr} \Gamma_{\widehat h}=\bigcap_{e\in h^{-1}(z_0)} \Gamma_{h,e}.\ee Moreover, it is clear that  $\widehat{h}=h\circ t$ for some branched covering map $t\colon \widehat{R}\to R$.

Let now $h'=h\circ t'$ be a Galois covering.
Since $S_h\subseteq S_{h'}$, we may choose
$z_0\in C\setminus S_{h'}$ and regard all covering
subgroups as subgroups of
$\pi_1(C\setminus S_{h'},z_0)$.
The equality defining $\Gamma_{\h h}$ remains valid
after adding these punctures.
Since $\Gamma_{h'}$ is normal, the inclusion
$\Gamma_{h'}\subseteq\Gamma_{h,e}$ for some
$e\in h^{-1}(z_0)$ implies this inclusion for every
$e\in h^{-1}(z_0)$.
Hence, $\Gamma_{h'}\subseteq\Gamma_{\h h}$, which gives
a branched covering map $u:R'\to\h R$ satisfying
$h'=\h h\circ u$.

 Uniqueness up to equivalence follows from the universal property:
any two normalizations factor through each other, so the
corresponding maps between their source surfaces have degree one.
  
Finally, a point $z\in C$ belongs to $S_h$ if and only if a small loop $\gamma\in \pi_1(C\setminus S_h, z_0)$ around $z$ does not belong to $\Gamma_{h,e}$ for some $e$ or, equivalently, does not belong to the group \eqref{inr}. Since the same condition determines points $z$ that belong to $S_{\widehat h}$, we conclude that $S_{\widehat h}=S_h$. 
 \qed


The following two lemmas show that fiber products of maps, one of which is Galois, possess special properties that make them easier to handle.

\begin{lemma}
Let $f:C_1\to C$ and $g:C_2\to C$ be branched covering
maps between closed surfaces. Assume that $g$ is a
Galois covering. Then the maps $h_j:R_j\to C$
corresponding to the components of the fiber product
of $f$ and $g$ are equivalent. Moreover, if $f$ is
also a Galois covering, then all these maps are
Galois coverings.
\end{lemma}

\begin{proof}
For $z_0\in C\setminus S_{f,g}$ and
$a_1,a_2\in f^{-1}(z_0)$, choose
$\gamma\in\pi_1(C\setminus S_{f,g},z_0)$ with
\[
\gamma\Gamma_{f,a_1}\gamma^{-1}=\Gamma_{f,a_2}.
\]
Then, because $\Gamma_g$ is normal,
\[
\gamma(\Gamma_{f,a_1}\cap\Gamma_g)\gamma^{-1}
=\Gamma_{f,a_2}\cap\Gamma_g.
\]
Hence the corresponding covering maps $h_j:R_j\to C$
are equivalent.

Furthermore, if $f$ is also Galois, then the
intersection $\Gamma_f\cap\Gamma_g$ is normal in
$\pi_1(C\setminus S_{f,g},z_0)$, so all the maps $h_j$
are Galois coverings.
\end{proof}

\bl \l{12} 
Let $f\colon C_1\rightarrow C$ and $g\colon C_2\rightarrow C$ be branched covering maps between closed surfaces. Assume that $g$ is a Galois covering. Then for any $z_0\in C\setminus S_{f,g}$ and $a\in f^{-1}(z_0)$, the subgroups $\Gamma_{f,a}$ and $\Gamma_g$ are permutable, and the product subgroups $\Gamma_{f,a}\Gamma_g,$  $a\in f^{-1}(z_0),$ are conjugate.
\el 
\pr 
For any $\gamma\in\Gamma_{f,a}$ and $\delta\in\Gamma_g,$ we have 
\[
\gamma\delta = (\gamma\delta\gamma^{-1})\gamma \in \Gamma_g\Gamma_{f,a},
\]
and similarly,
\[
\delta\gamma = \gamma(\gamma^{-1}\delta\gamma)\in\Gamma_{f,a}\Gamma_g.
\]
Thus, the subgroups $\Gamma_{f,a}$ and $\Gamma_g$ are permutable. 

Furthermore, if $a_1,a_2\in f^{-1}(z_0),$ then there exists $\gamma\in\pi_1(C\setminus S_{f,g},z_0)$ such that $$\gamma\Gamma_{f,a_1}\gamma^{-1}=\Gamma_{f,a_2}.$$ Since $\Gamma_g$ is a normal subgroup, it follows that
\[
\gamma(\Gamma_{f,a_1}\Gamma_g)\gamma^{-1} = (\gamma\Gamma_{f,a_1}\gamma^{-1})(\gamma\Gamma_g\gamma^{-1}) = \Gamma_{f,a_2}\Gamma_g. \eqno{\Box}
\]

For branched covering maps $f$ and $g$, where $g$ is a Galois covering, and \linebreak $z_0\in C\setminus S_{f, g}$, we denote by $f\star g$ any representative of the equivalence class of branched coverings corresponding to the subgroup $\Gamma_{f,a}\Gamma_g$.

\bt \la{kij} Let $f: C_1\rightarrow C$ and $g: C_2\rightarrow C$ be branched covering maps between closed surfaces.  Then the equality 
\be \l{thee} o(f,g)=o(f\star \widehat g,\;g)\ee
holds.
\et

\pr 
Set
$
G=\pi_1(C\setminus S_{f,g},z_0)
$
and consider its diagonal action on
\[
X=f^{-1}(z_0)\times g^{-1}(z_0).
\]
Identifying the fibers of $f$ and $g$ with the corresponding sets
of left cosets, we identify $X$ with the set of pairs
\[
\bigl(
\alpha_{j_1}\Gamma_{f,a},\, 
\beta_{j_2}\Gamma_{g,b}
\bigr),
\qquad
1\leq j_1\leq n,\quad
1\leq j_2\leq m,
\]
where $n=\deg f$ and $m=\deg g$. Two such pairs
\[
\bigl(
\alpha_{j_1}\Gamma_{f,a},\,
\beta_{j_2}\Gamma_{g,b}
\bigr)
\quad\text{and}\quad
\bigl(
\alpha_{i_1}\Gamma_{f,a},\,
\beta_{i_2}\Gamma_{g,b}
\bigr)
\]
belong to the same orbit if and only if there exists
$\gamma\in G$ such that
\[
\gamma\alpha_{j_1}\Gamma_{f,a}
=
\alpha_{i_1}\Gamma_{f,a}
\quad\text{and}\quad
\gamma\beta_{j_2}\Gamma_{g,b}
=
\beta_{i_2}\Gamma_{g,b}.
\]
Equivalently, this holds if and only if the set
\be \l{cfr}
\alpha_{i_1}\Gamma_{f,a}\alpha_{j_1}^{-1}
\cap
\beta_{i_2}\Gamma_{g,b}\beta_{j_2}^{-1}
\ee
is non-empty.

We next consider the diagonal monodromy action on
\[
Y=(f\star\h g)^{-1}(z_0)\times g^{-1}(z_0).
\]
This action is naturally defined for the group
\[
G'=\pi_1(C\setminus S_{f\star\h g,g},z_0).
\]
Since $f\star\h g$ is a compositional left factor of $f$, we have
$$
S_{f\star\h g,g}\subseteq S_{f,g}.
$$ 
Hence, the inclusion
\[
C\setminus S_{f,g}
\longrightarrow
C\setminus S_{f\star\h g,g}
\]
induces a natural epimorphism
$
\rho:G\longrightarrow G'.
$
Composing the action of $G'$ on $Y$ with $\rho$, we obtain an
action of $G$ on $Y$. Since $\rho$ is surjective, the orbits of
this action of $G$ are exactly the orbits of the original action
of $G'$. 
The subgroup of $G$ corresponding to the covering
$f\star\h g$ is $\Gamma_{f,a}\Gamma_{\h g}$. Therefore, after
identifying the fibers with sets of left cosets, the action of
$G$ on $Y$ is identified with its diagonal action on the pairs
\[
\bigl(
\alpha_{j_1}\Gamma_{f,a}\Gamma_{\h g},\, 
\beta_{j_2}\Gamma_{g,b}
\bigr).
\]

We now define a map $\varphi$ from the orbits of $G$ on $X$ to
the orbits of $G$ on $Y$. Namely, the orbit containing
\[
\bigl(
\alpha_{j_1}\Gamma_{f,a},\,
\beta_{j_2}\Gamma_{g,b}
\bigr)
\]
is mapped to the orbit containing
\[
\bigl(
\alpha_{j_1}\Gamma_{f,a}\Gamma_{\h g},\,
\beta_{j_2}\Gamma_{g,b}
\bigr).
\]

If the set \eqref{cfr} is non-empty, then the set
\be \l{cfr1}
\alpha_{i_1}\Gamma_{f,a}\Gamma_{\h g}\alpha_{j_1}^{-1}
\cap
\beta_{i_2}\Gamma_{g,b}\beta_{j_2}^{-1}
\ee
is also non-empty. Thus, $\phi$ is well-defined. Moreover, $\phi$ is surjective, since every left coset of
$\Gamma_{f,a}\Gamma_{\widehat g}$ contains a left coset
of $\Gamma_{f,a}$.

To prove the injectivity of $\phi$, we must show that if the set
\eqref{cfr1} is non-empty, then the set \eqref{cfr} is also
non-empty. Suppose that $x$ is an element of \eqref{cfr1}.
Since $\Gamma_{\h g}$ is normal, we have
\[
\alpha_{i_1}\Gamma_{f,a}\Gamma_{\h g}\alpha_{j_1}^{-1}
=
\alpha_{i_1}\Gamma_{f,a}\alpha_{j_1}^{-1}\Gamma_{\h g}.
\]
Therefore, $x$ can be represented in the form
\[
x
=
\alpha_{i_1}\alpha\alpha_{j_1}^{-1}\beta
=
\beta_{i_2}\gamma\beta_{j_2}^{-1}
\]
for some $\alpha\in\Gamma_{f,a}$,
$\beta\in\Gamma_{\h g}$, and $\gamma\in\Gamma_{g,b}$.

Furthermore, since
\[
\beta_{j_2}^{-1}\beta\beta_{j_2}
\in
\Gamma_{\h g}
\subseteq
\Gamma_{g,b},
\]
there exists $\gamma_1\in\Gamma_{g,b}$ such that
\[
\beta=\beta_{j_2}\gamma_1\beta_{j_2}^{-1}.
\]
Setting $y=x\beta^{-1}$, we obtain
\[
y
=
\alpha_{i_1}\alpha\alpha_{j_1}^{-1}
=
\beta_{i_2}\gamma\beta_{j_2}^{-1}\beta^{-1}
=
\beta_{i_2}\gamma\gamma_1^{-1}\beta_{j_2}^{-1}.
\]
Thus, $y$ belongs to the set \eqref{cfr}, and hence this set is
non-empty. 
\qed

While the condition that branched covering maps $f$ and $g$ have no non-trivial common compositional left factor in general does not imply that the fiber product of $f$ and $g$ is irreducible, Theorem \ref{kij} implies the following corollary.  

\bc \l{such} 
Let $f: C_1\rightarrow C$ and $g: C_2\rightarrow C$ be branched covering maps  between closed surfaces such that $f$ and $\widehat g$ have no non-trivial common compositional left factor. Then the fiber product of $f$ and $g$ is irreducible.
\ec 
\pr By Theorem \ref{kij}, equality \eqref{thee}
holds. On the other hand, since $f$ and $\widehat g$ have no non-trivial common left compositional factor,  
$$\Gamma_{f,a}\Gamma_{\widehat g}=\langle\Gamma_{f,a},\Gamma_{\widehat g}\rangle=\pi_1(C\setminus S_{f,g},z_0),$$ 
by Proposition \ref{p67}. Thus, 
$f\star\widehat g$ is equivalent to the identity map, and for such a map, the fiber product with any map is irreducible by Proposition \ref{p66}.\qed

Another corollary of 
Theorem \ref{kij} is the following statement.  

\bc 
Let $f: C_1\rightarrow C$ and $g: C_2\rightarrow C$ be branched covering maps between closed surfaces. Assume that $g$ is a Galois covering. Then the fiber product of $f$ and $g$ is reducible if and only if $f$ and $g$ have a non-trivial common left compositional factor. 
\ec 
\pr
Since for a Galois covering $g$, we have $\widehat g = g$,  the ``only if'' part follows from  Corollary \ref{such}.  The ``if'' part follows from Corollary \ref{kaka}. \qed

Let $h:\, R\rightarrow C$ be a branched covering map  of degree at least two.   The map 
$h$ is called {\it indecomposable} if the equality $h = h_2 \circ h_1$, where $h_1:\, R\rightarrow R_1,$ $h_2:\, R_1\rightarrow C$ are branched covering maps, implies that at least one of the maps $h_1, h_2$ is of degree one. It is clear that any $h$ as above admits a representation
of the form $$h = h_r \circ h_{r-1} \circ \dots \circ h_1,$$ where $h_1, h_2, \dots, h_r$ are indecomposable branched covering maps of degree at least two. It is also clear that 
$h$ is indecomposable if and only if for any $z_0\in C\setminus S_h$ the groups $\Gamma_{h,c}$, $c\in h^{-1}(z_0)$, are maximal subgroups of $\pi_1(C\setminus S_h, z_0).$ 
Note that a description of all possible decompositions of a covering is a highly non-trivial question, and even for rational functions  a complete answer is known when $h$ is a polynomial \cite{r1} or a Laurent polynomial \cite{pak}. 

The following result generalizes the corresponding result of Fried about irreducibility of algebraic curves \eqref{cura}, where $f$ and $g$ are rational
functions (see \cite{f2}, Proposition 2).

\bt \la{p6} 
Let $f: C_1\rightarrow C$ and $g: C_2\rightarrow C$ be branched covering maps between closed surfaces such that $o(f,g)>1.$ Then there exist branched covering maps  between closed surfaces  $f_1: \tilde{C}_{1}\rightarrow C,$ $g_1: \tilde{C}_{2}\rightarrow C,$  and $p: C_{1}\rightarrow \tilde{C}_{1},$ $q: C_{2}\rightarrow \tilde{C}_{2}$ such that: 
\begin{enumerate}[label=\upshape(\alph*)]
    \item The equalities $f = f_1 \circ p$ and $g = g_1 \circ q$ hold. 
   \item The maps $\widehat{f}_1$ and $\widehat{g}_1$ are equivalent.
    \item The equality $o(f,g) = o(f_1, g_1)$ holds. 
    
\end{enumerate}
\et
\pr For a branched covering map $h:\, R\rightarrow C$, denote by  
$d(h)$ the maximal number such that there
exists a decomposition of $h$ into a composition of $d(h)$ indecomposable branched covering maps of degree at least two. 
We proceed by induction on the number $d=d(f)+d(g)$.

If $d=2$, that is, if both maps 
$f,g$ are indecomposable, then Lemma \ref{12} and the maximality of $\Gamma_{f,a}$ imply  
that  either \be \la{ry} \Gamma_{f,a}\Gamma_{\widehat g}=\Gamma_{f,a}, \qquad a\in f^{-1}\{z_0\},\ee 
or \be \la{ryry} \Gamma_{f,a}\Gamma_{\widehat g}=\pi_1(C\setminus S_{f,g},z_0), \qquad a\in f^{-1}\{z_0\}.\ee
However, in the latter case Proposition \ref{p67} and Corollary \ref{such} yield that $o(f,g)=1$, in contradiction to the assumption. Therefore, 
equalities \eqref{ry} hold, and hence
$$\Gamma_{\widehat g}\subseteq\bigcap_{a\in f^{-1}\{z_0\}} \Gamma_{f,a}= \Gamma_{\widehat f} .$$ Similarly,  $\Gamma_{\widehat f} \subseteq \Gamma_{\widehat g}.$
Thus, $\Gamma_{\widehat g}= \Gamma_{\widehat f}$, and we can set $f_1=f,$ $g_1=g.$

Suppose now that $d>2.$
If $\Gamma_{\widehat f} =\Gamma_{\widehat g}$, then as above we can set  
$f_1=f,$ $g_1=g$, so assume that $\Gamma_{\widehat f} \neq \Gamma_{\widehat g}$. Then,   either  
\be \la{xzc} \Gamma_{f,a}\subsetneq \Gamma_{f,a}\Gamma_{\widehat g}, \qquad a\in f^{-1}\{z_0\},\ee
or 
$$\Gamma_{g,b}\subsetneq \ \Gamma_{g,b}\Gamma_{\widehat f}, \qquad b\in g^{-1}\{z_0\}.$$

Suppose, say, that \eqref{xzc} holds. Then
\be \l{kon} 
f=(f\star\widehat g)\circ p
\ee
for some branched covering map $p$ with $\deg p>1$.
Moreover, since equality \eqref{ryry} is impossible, we have
$\deg(f\star\widehat g)>1$.
The  decomposition \eqref{kon} implies the inequality $d(f\star\widehat g)<d(f)$.
Since Theorem~6.13 yields
\[
o(f\star\widehat g,g)=o(f,g)>1,
\]
we can apply the induction hypothesis to 
$f\star\widehat g$ and $g$, and the theorem follows. \qed

\end{document}